\pdfoutput=1
\documentclass[a4paper,11pt,oneside]{amsart}
\usepackage[paperheight=279mm,paperwidth=18cm,textheight=26cm,textwidth=14cm]{geometry}
\usepackage{amsmath,amsfonts,amssymb,amsthm,mathtools}
\usepackage[utf8]{inputenc}
\usepackage[T1]{fontenc}
\usepackage[unicode,hidelinks,bookmarks]{hyperref}
\hypersetup{
colorlinks=true,
citecolor=[rgb]{0,0,0.5},
linkcolor=[rgb]{0,0,0.5},
urlcolor=[rgb]{0,0,0.75},
pdfpagemode=UseNone,
pdfstartview=FitH,
pdfdisplaydoctitle=true,
pdftitle={Maximal polynomial modulations},
pdfauthor={Pavel Zorin-Kranich},
pdflang=en-US
}

\usepackage[style=alphabetic]{biblatex}
\DeclarePairedDelimiter\abs{\lvert}{\rvert}
\DeclarePairedDelimiter\meas{\lvert}{\rvert}
\DeclarePairedDelimiter\norm{\lVert}{\rVert}
\providecommand\given{}
\newcommand\SetSymbol[1][]{%
\nonscript\:#1\vert
\allowbreak
\nonscript\:
\mathopen{}}
\DeclarePairedDelimiterX\Set[1]\{\}{%
\renewcommand\given{\SetSymbol[\delimsize]}
#1
}

\numberwithin{equation}{section}
\newtheorem{theorem}[equation]{Theorem}
\newtheorem{corollary}[equation]{Corollary}
\newtheorem{lemma}[equation]{Lemma}
\newtheorem{proposition}[equation]{Proposition}

\theoremstyle{definition}
\newtheorem{definition}[equation]{Definition}
\theoremstyle{remark}

\newtheorem{remark}[equation]{Remark}
\newtheorem*{ack}{Acknowledgment}

\newcommand{\calD}{\mathcal{D}}
\newcommand{\calL}{\mathcal{L}}
\newcommand{\calQ}{\mathcal{Q}}
\newcommand{\Rmin}{\underline{R}}
\newcommand{\Rmax}{\overline{R}}

\renewcommand{\top}{\mathrm{top}}
\newcommand{\calE}{\mathcal{E}}
\newcommand{\calF}{\mathcal{F}}
\newcommand{\calS}{\mathcal{S}}
\newcommand{\calJ}{\mathcal{J}}
\newcommand{\calH}{\mathcal{H}}
\newcommand{\calC}{\mathcal{C}}
\newcommand{\calN}{\mathcal{N}}
\newcommand{\ch}{\operatorname{ch}}
\newcommand{\scale}{s}
\newcommand{\gen}{\operatorname{gen}}
\newcommand{\smin}{\underline{\sigma}}
\newcommand{\smax}{\overline{\sigma}}
\newcommand{\sumin}{\scale_{\min}}
\newcommand{\sumax}{\scale_{\max}}
\newcommand{\bd}{\operatorname{bd}}
\newcommand{\ds}{\mathbf{d}} 
\def\C{\mathbb{C}}
\newcommand{\CZK}{K}
\newcommand{\R}{\mathbb{R}}
\newcommand{\N}{\mathbb{N}}
\newcommand{\Z}{\mathbb{Z}}
\newcommand{\dif}{\mathrm{d}}
\DeclareMathOperator{\dist}{dist}
\DeclareMathOperator{\diam}{diam}
\DeclareMathOperator{\supp}{supp}
\DeclareMathOperator{\mdens}{\overline{dens}}
\newcommand{\mE}{\overline{E}}
\newcommand{\one}{\mathbf{1}}

\def\PZdefchar#1{\expandafter\def\csname T#1\endcsname{\mathfrak{#1}}}
\def\PZdefloop#1{\ifx#1\PZdefloop\else\PZdefchar#1\expandafter\PZdefloop\fi}
\PZdefloop abcdefghijklmnopqrstuvwxyzABCDEFGHIJKLMNOPQRSTUVWXYZ\PZdefloop

\allowdisplaybreaks
\begin{document}
\title[Maximal polynomial modulations]{Maximal polynomial modulations\\ of singular integrals}
\author{Pavel Zorin-Kranich}
\address{University of Bonn\\
  Mathematical Institute\\
  Bonn\\
  Germany
}
\begin{abstract}
Let $\CZK$ be a standard H\"older continuous Calder\'on--Zygmund kernel on $\R^{\ds}$ whose truncations define $L^{2}$ bounded operators.
We show that the maximal operator obtained by modulating $\CZK$ by polynomial phases of a fixed degree is bounded on $L^{p}(\R^{\ds})$ for $1 < p < \infty$.
This extends Sj\"olin's multidimensional Carleson theorem and Lie's polynomial Carleson theorem.
\end{abstract}
\maketitle

\section{Introduction}

This article continues a line of research in time-frequency analysis started with Carleson's theorem on pointwise almost everywhere convergence of Fourier series of $L^{2}$ functions on $\R/\Z$ \cite{MR0199631}.
In view of Stein's maximal principle \cite{MR125392}, this result is equivalent to the weak type $(2,2)$ bound for an associated maximal operator, called the \emph{Carleson operator}.
Two essentially different approaches to $L^{p}$ bounds for this operator were introduced in \cite{MR0340926,MR1783613}.

There are several possible analogues of Carleson's theorem in higher dimensions, depending on the chosen generalization of the interval multipliers.
One direction concerns Bochner--Riesz summation, which becomes necessary because the ball multiplier is unbounded on any $L^{p}(\R^{\ds})$ space for $\ds\geq 2$ and $p\neq 2$ \cite{MR0296602}.
Bochner--Riesz summation is embedded in a network of open problems centered around the so-called \emph{local smoothing conjecture}.
We refer to \cite{arxiv:1812.11616} for a recent survey of this topic and remark that this conjecture has been solved in dimension $\ds=2$ in \cite{MR4151084}.

We are concerned with a different direction, where the Hilbert transform appearing in connection with Fourier summation is replaced by a more general singular integral.
A first result of this kind, in which the singular integral is maximally modulated by plane waves, appeared in \cite{MR0336222}.
Polynomial modulations of singular integrals were studied in \cite{MR890662} in connection with analysis on nilpotent Lie groups.
This led to Stein's conjecture that \emph{maximal polynomial} modulations of singular integrals define $L^{p}$ bounded operators, backed by a proof of concept result in \cite{MR1364908} involving \emph{maximal quadratic} modulations.
A more general result, involving maximal modulations by polynomials without linear terms, was obtained by more flexible methods in \cite{MR1879821}.
Decisive progress on Stein's conjecture was made in \cite{MR2545246}, where boundedness of the Hilbert transform maximally modulated by polynomial phases involving both linear and quadratic terms was proved.
That result was later extended to polynomials of an arbitrary fixed degree and the full range of $L^{p}$ spaces in \cite{MR4125450}.
In this article, drawing most heavily on the ideas introduced in \cite{MR0340926,MR1879821,MR4125450}, we prove Stein's conjecture in dimensions $\ds > 1$.

Our result is most conveniently formulated as a uniform estimate for truncated singular integrals.
We begin with the necessary notation.
Let $\CZK$ be a $\tau$-H\"older continuous Calder\'on--Zygmund (CZ) kernel on $\R^{\ds}$, where $\tau>0$ and $\ds\geq 1$, that is, a function $\CZK:\Set{(x,y) \in \R^{\ds}\times\R^{\ds} \given x\neq y} \to \C$ such that
\begin{equation}
\label{eq:K-size}
\abs{\CZK(x,y)} \lesssim \abs{x-y}^{-\ds}
\text{ if } x \neq y \text{ and}
\end{equation}
\begin{equation}
\label{eq:K-reg}
\abs{\CZK(x,y)-\CZK(x',y)}+\abs{\CZK(y,x)-\CZK(y,x')} \lesssim \frac{\abs{x-x'}^{\tau}}{\abs{x-y}^{\ds+\tau}}
\text{ if }
\abs{x-x'} < \frac12 \abs{x-y}.
\end{equation}
Suppose that the associated truncated integral operators
\[
T_{\Rmin}^{\Rmax}f(x) := \int_{\Rmin<\abs{x-y}<\Rmax} \CZK(x,y) f(y) \dif y
\]
are bounded on $L^{2}(\R^{\ds})$ uniformly in $0<\Rmin<\Rmax<\infty$ (this condition follows from boundedness of a CZ operator associated to $\CZK$ on $L^{2}(\R^{\ds})$).
We define the associated maximally polynomially modulated, maximally truncated singular integral operator by
\begin{equation}
\label{eq:T}
Tf(x):=\sup_{Q\in\calQ_{d}} \sup_{0< \Rmin \leq \Rmax < \infty}
\abs[\Big]{\int_{\Rmin \leq \abs{x-y} \leq \Rmax} \CZK(x,y) e^{2\pi i Q(y)} f(y) \dif y}
\end{equation}
for $f\in L^{1}_{\mathrm{loc}}(\R^{\ds})$, where $d\geq 1$ and $\calQ_{d}$ denotes the set of all polynomials in $\ds$ variables with real coefficients and degree at most $d$.

\begin{theorem}
\label{thm:main}
The operator \eqref{eq:T} is bounded on $L^{p}(\R^{\ds})$ for every $1 < p < \infty$.
\end{theorem}

Theorem~\ref{thm:main} extends several previous results.
\begin{enumerate}
\item Carleson's theorem \cite{MR0199631,MR0238019} is the case $\ds=d=1$, $\CZK(x,y)=1/(x-y)$ (alternative proofs are due to C.~Fefferman \cite{MR0340926} and Lacey and Thiele \cite{MR1783613}).
\item Sj\"olin's multidimensional Carleson theorem \cite{MR0336222,MR1781088} is the translation invariant case $\CZK(x,y)=\CZK(x-y)$ with $d=1$ (see also \cite{MR2007237,MR2031458} for an alternative proof using methods from \cite{MR1783613}).
\item Ricci and Stein's oscillatory singular integrals \cite{MR890662} arise if $\sup_{Q}$ is replaced by $Q=Q_{x}$ that itself depends polynomially on $x$ (see also \cite{arxiv:1701.05249} for sparse bounds).
\item Stein and Wainger \cite{MR1364908,MR1879821} restricted the supremum over $Q$ in such a way as to eliminate modulation invariance by linear phases.
\item V.~Lie \cite{MR2545246,arXiv:1105.4504} proved the general polynomial case $d\geq 1$ with $\CZK(x,y)=1/(x-y)$ in dimension $\ds=1$.
\item A non-translation invariant extension of Carleson's theorem was considered in \cite{MR2654142}.
\end{enumerate}

By the extrapolation argument introduced in \cite{MR3148061} (see Appendix~\ref{sec:extrapolation} for details), Theorem~\ref{thm:main} is a consequence of the following localized $L^{2}$ estimates.
\begin{theorem}
\label{thm:loc}
Let $0 \leq \alpha < 1/2$ and $0 \leq \nu,\kappa < \infty$.
Let $F,G \subset\R^{\ds}$ be measurable subsets and $\tilde{F} := \Set{ M\one_{F} > \kappa }$, $\tilde{G} := \Set{ M\one_{G} > \nu }$, where $M$ denotes the Hardy--Littlewood maximal operator.
Then
\begin{align}
\norm{T}_{2 \to 2} &\lesssim 1, \label{eq:loc:full}\\
\norm{\one_{G} T \one_{\R^{\ds}\setminus \tilde{G}}}_{2 \to 2} &\lesssim_{\alpha} \nu^{\alpha}, \label{eq:loc:G}\\
\norm{\one_{\R^{\ds} \setminus \tilde{F}} T \one_{F} }_{2 \to 2} &\lesssim_{\alpha} \kappa^{\alpha}. \label{eq:loc:F}
\end{align}
\end{theorem}
The estimate \eqref{eq:loc:full} is a special case of both \eqref{eq:loc:G} and \eqref{eq:loc:F}, but we formulate and prove it separately because it is the easiest case.

The estimate \eqref{eq:loc:G} is used in the range $2<p<\infty$.
It is also possible to reduce Theorem~\ref{thm:main} in this range to the case $p=2$.
Indeed, it can be shown using known techniques that, for any $1\leq p<\infty$, an unweighted weak type $(p,p)$ estimate for the operator \eqref{eq:T} implies that this operator can be dominated by sparse operators with $L^{p}$ means (see \cite[Theorem 4.3.2]{beltran-phdthesis} and \cite{MR3484688} for the shortest available proof of this implication).
This in turn implies strong type $(\tilde p,\tilde p)$ estimates (even vector-valued \cite{arxiv:1709.09647} and with a certain class of Muckenhoupt weights) for all $p<\tilde p<\infty$.
The observation that weighted estimates for maximally modulated singular integrals can be obtained using unweighted estimates as a black box by essentially the same argument as without the modulations goes back to \cite{MR0338655} and was expounded in \cite{MR2115460,MR3291794,arxiv:1410.6085,arxiv:1611.03808}.

Since the above discussion shows that the strength of Theorem~\ref{thm:main} decreases with $p$, it is unsurprising that \eqref{eq:loc:G} can be obtained by a minor variation of the proof of \eqref{eq:loc:full}.
Nevertheless, we hope that the simplicity of this localized estimate can motivate the more difficult localization argument required to prove the estimate \eqref{eq:loc:F}, that is used in the range $1 < p < 2$.

The following ingredients of our proof have appeared in previous works.
\begin{enumerate}
\item The overall structure of the argument (in particular the decomposition into trees, the selection algorithm in Section~\ref{sec:tree-selection}, the single tree estimate, and the splitting into rows) is due to C.~Fefferman \cite{MR0340926}.
\item The discretization of the space of polynomials has the same properties (parts \ref{def:tile:ball} and \ref{def:tile:cover} of Lemma~\ref{lem:tile}) as in \cite{MR2545246,arXiv:1105.4504}.
\item The iteration of the Fefferman selection algorithm between stopping times as in Lemma~\ref{lem:spatial-decomposition} and the associated spatial orthogonality argument in Section~\ref{sec:orth-stop-gen} have been introduced in \cite{arXiv:1105.4504}.
This is the main tool that allows to obtain $L^{2} \to L^{2}$ estimates directly (without interpolation with $L^{p}$, $p<2$).
Earlier arguments, starting with \cite{MR0340926}, only use $L^{\infty}$-forests of generation $k=0$ (see Definition~\ref{def:F-forest}).
\item The selection of and estimates for antichains and boundary parts of trees in Proposition~\ref{prop:fef-forest} and Proposition~\ref{prop:sf}, respectively, are adapted from \cite{arXiv:1105.4504}.
\item The extrapolation of localized $L^{2}$ estimates to $L^{p}$ estimates was found by Bateman in connection with the directional Hilbert transform \cite{MR3148061}.
\item For the usual Carleson operator (case $d=\ds=1$), the localized estimates in Theorem~\ref{thm:loc} are contained in \cite[estimate (76) in arxiv version 2]{arxiv:1707.05484} (more generally, that article also deals with the $r$-variational Carleson operator, in which case the range of $\alpha$ also depends on the variational exponent $r$).
A different approach to localization can be found in \cite{arxiv:1612.03028}.
\end{enumerate}

The following elements are new in this context.
\begin{enumerate}
\item In Lemma~\ref{lem:spatial-decomposition}, we use a single stopping time for all densities.
  This helps to ensure that all trees in the decomposition \eqref{eq:tree-dec} are convex (unlike the version of the argument from \cite{arXiv:1105.4504} explained in \cite{MR3334208}).
  Also, we consider all dyadic scales at once rather than splitting them in congruence classes modulo a large integer.
  This is crucial for general CZ kernels (that do not satisfy a cancellation condition), since removing some scales from a general CZ operator can destroy its $L^{2}$ boundedness.
\item Our tiles are nested both in space and in frequency (part \ref{def:tile:nested} of Lemma~\ref{lem:tile}), similarly to \cite{MR0340926} and differently from \cite{MR2545246,arXiv:1105.4504}.
This is achieved using a variant of the Christ grid cubes construction and simplifies the combinatorics of tiles.
\item We estimate oscillatory integrals using a single scale van der Corput type estimate (Lemma~\ref{lem:osc-int}, adapted from \cite{MR1879821}).
This allows us to substantially reduce the regularity hypothesis on the kernel $\CZK$ compared to the previous works in which this issue was raised \cite{MR0336222,arxiv:1710.10962}.
\item We use the $L^{2}(\R^{\ds})$ boundedness of truncated operators associated to $\CZK$ as a black box.
This hypothesis can be verified for example using a $T(b)$ theorem.
\item We apply the extrapolation idea from \cite{MR3148061} in the context of a Fefferman type argument for the Carleson operator.
The required localized estimate is obtained by an argument that resembles the single tree estimate in \cite{MR1783613}.
Specifically, in Lemma~\ref{lem:tree} we obtain sharp decay and in Proposition~\ref{prop:sf-loc} almost sharp decay in both localization parameters.
\end{enumerate}

It appears plausible that our proof should also work for CZ kernels adapted to an anisotropic group of dilations (see \cite{arxiv:1710.10962} for a recent result in this setting) using a discretization based on Christ grid cubes \cite{MR1096400} also in space.

Two different approaches to $L^{p}$ estimates for the (polynomial) Carleson operator in the range $1<p<2$ appear in \cite{arXiv:1105.4504,arXiv:1712.03092} and in \cite{MR3148602}.
Our approach is closer to the latter, and it seems possible to obtain Lorentz space estimates near $L^{1}$ combining our arguments with the ideas in \cite{MR3148602}.
However, I have not been able to recover the best known estimates for the Carleson operator in this way.

\begin{ack}
I thank
\begin{enumerate}
\item Victor Lie and Stefan Oberd\"orster for pointing out a few errors in previous versions of this article,
\item Stefan Ober\-d\"orster, João Pedro Ramos, Olli Saari, and Christoph Thiele for participating in a study group and supplying numerous suggestions for simplification and improvement of exposition, and
\item Camil Muscalu for pointing out the connection between the localized estimates in Theorem~\ref{thm:loc} and the helicoidal method.
\item The anonymous referees for supplying several suggestions for improvement of exposition.
\end{enumerate}
The author was partially supported by DFG SFB 1060 and the Hausdorff Center for Mathematics in Bonn (DFG EXC 59).
\end{ack}

\section{Discretization}
\label{sec:discretization}
Modifying the notation used in the introduction, we denote by $\calQ$ the vector space of all real polynomials in $\ds$ variables of degree at most $d$ modulo $+\R$.
That is, we identify two polynomials if and only if their difference is constant.
This identification is justified by the fact that the absolute value of the integral in \eqref{eq:T} does not depend on the constant term of $Q$.
Notice that $Q(x)-Q(x')\in\R$ is well-defined for $Q\in\calQ$ and $x,x'\in\R^{\ds}$.

Let $D=D(d,\ds)$ be a large integer to be chosen later.
Let $\psi$ be a smooth function supported on the interval $[1/(4D),1/2]$ such that $\sum_{s\in\Z} \psi(D^{-s}\cdot) \equiv 1$ on $(0,\infty)$.
Then the kernel can be decomposed as
\[
\CZK(x,y) = \sum_{s\in\Z} \CZK_s(x,y)
\text{ with }
\CZK_{s}(x,y) := \CZK(x,y) \psi(D^{-s} \abs{x-y}).
\]
The functions $\CZK_{s}$ are supported on the sets $\Set{(x,y)\in\R^{\ds}\times\R^{\ds} \given D^{s-1}/4<\abs{x-y}<D^{s}/2}$ and satisfy
\begin{equation}
\label{eq:Ks-size}
\abs{\CZK_{s}(x,y)} \lesssim D^{-\ds s}
\text{ for all } x, y \in \R^{\ds},
\end{equation}
\begin{equation}
\label{eq:Ks-reg}
\abs{\CZK_{s}(x,y)-\CZK_{s}(x',y)}+\abs{\CZK_{s}(y,x)-\CZK_{s}(y,x')} \lesssim \frac{\abs{x-x'}^{\tau}}{D^{(\ds+\tau)s}}
\text{ for all } x, x', y \in \R^{\ds}.
\end{equation}
We can replace the maximal operator \eqref{eq:T} by the smoothly truncated operator
\begin{equation}
\label{eq:T-smooth-trunc}
Tf(x):=\sup_{Q\in\calQ_{d}} \sup_{\smin \leq \smax \in \Z}
\abs[\Big]{ \sum_{s=\smin(x)}^{\smax(x)} \int \CZK_{s}(x,y) e(Q(y)) f(y) \dif y},
\end{equation}
where $e(t)=e^{2\pi i t}$ denotes the standard character on $\R$,
at the cost of an error term that is controlled by the Hardy--Littlewood maximal operator $M$ (the required localized estimates for $M$ are easy, see Lemma~\ref{lem:M-loc}).

Since the absolute value of the integral in \eqref{eq:T-smooth-trunc} is a continuous function of $Q$, we may restrict $\smin,\smax,Q$ to a finite set as long as we prove estimates that do not depend on this finite set.
After these preliminary reductions, we can linearize the supremum in \eqref{eq:T-smooth-trunc} and replace that operator by
\begin{equation}
\label{eq:T-linearized}
Tf(x):=\sum_{s=\smin(x)}^{\smax(x)} \int \CZK_{s}(x,y) e(Q_{x}(x)-Q_{x}(y)) f(y) \dif y,
\end{equation}
where $\smin,\smax : \R^{\ds}\to\Z$, $Q_{\cdot}:\R^{\ds}\to\calQ$ are measurable functions with finite range.
Let $\sumin := \min_{x\in\R^{\ds}} \smin(x) > -\infty$ and $\sumax := \max_{x\in\R^{\ds}} \smax(x) < +\infty$.
All stopping time constructions will start at the largest scale $\sumax$ and terminate after finitely many steps at the smallest scale $\sumin$.

\subsection{Tiles}
\label{sec:tiles}
The grid of $D$-adic cubes in $\R^{\ds}$ will be denoted by
\[
\calD := \bigcup_{s \in \Z} \calD_{s},
\quad
\calD_{s} := \Set[\big]{ \prod_{i=1}^{\ds}[D^{s}a_{i},D^{s}(a_{i}+1)) \given a_{1},\dotsc,a_{\ds}\in\Z}.
\]
We denote elements of $\calD$ by the letters $I$, $J$ and call them \emph{grid cubes}.
The unique integer $s=\scale(I)$ such that $I\in\calD_{s}$ will be called the \emph{scale} of a grid cube.
The \emph{parent} of a grid cube $I$ is the unique grid cube $\hat{I} \supset I$ with $\scale(\hat{I}) = \scale(I)+1$.
The side length of a cube $I$ is denoted by $\ell(I)$.
If $I$ is a cube and $a>0$, then $aI$ denotes the concentric cube with side length $a\ell(I)$.

For every bounded subset $I\subset\R^{\ds}$, we define a norm on $\calQ$ by
\begin{equation}
\label{eq:normI}
\norm{Q}_{I} := \sup_{x,x'\in I} \abs{Q(x)-Q(x')},
\quad
Q\in\calQ.
\end{equation}

\begin{lemma}
\label{lem:normQ}
If $Q\in\calQ$ and $B(x,r) \subset B(x,R) \subset \R^{\ds}$, then
\begin{align}
\norm{Q}_{B(x,R)}
&\lesssim_{d} \label{eq:normQ:up}
(R/r)^{d} \norm{Q}_{B(x,r)},\\
\norm{Q}_{B(x,r)}
&\lesssim_{d} \label{eq:normQ:low}
(r/R) \norm{Q}_{B(x,R)}.
\end{align}
\end{lemma}

\begin{proof}
By translation, we may assume $x=0$, and we choose a representative for the congruence class modulo $+\R$ with $Q(0)=0$.

To show \eqref{eq:normQ:up}, suppose by scaling that $r=1$ and $\norm{Q}_{B(x,r)}=1$.
The coefficients of $Q$ can now be recovered from its values on the unit ball using a multivariate Lagrange interpolation formula, see e.g.\ \cite[Theorem 3.1]{MR317511}.
In particular, these coefficients are $O_{d,\ds}(1)$, and the conclusion follows.

Similarly, to show \eqref{eq:normQ:low}, suppose by scaling that $R=1$ and $\norm{Q}_{B(x,R)}=1$.
Then the coefficients of $Q$ are again $O_{d,\ds}(1)$, and the conclusion follows.
\end{proof}

\begin{corollary}
\label{cor:normQ}
If $D$ is sufficiently large, then, for every $I \in \calD$ and $Q \in \calQ$, we have
\begin{equation}
\label{eq:normQ:parent}
\norm{Q}_{\hat{I}} \geq 10^{4} \norm{Q}_{I}.
\end{equation}
\end{corollary}

We choose $D$ so large that \eqref{eq:normQ:parent} holds.

\begin{definition}
\label{def:pair}
A \emph{pair} $\Tp$ consists of a \emph{spatial cube} $I_{\Tp}\in\calD$ and a Borel measurable subset $\calQ(\Tp) \subset \calQ$ that will be called the associated \emph{uncertainty region}.
Abusing the notation, we will say that $Q\in \Tp$ if and only if $Q\in \calQ(\Tp)$.
Also, $\scale(\Tp):=\scale(I_{\Tp})$.

\begin{lemma}
\label{lem:tile}
There exist collections of pairs $\TP_{I}$, indexed by the grid cubes $I\in\calD$ with $\sumin \leq \scale(I) \leq \sumax$, such that
\begin{enumerate}
\item\label{def:tile:ball} To each $\Tp \in \TP_{I}$ is associated a \emph{central polynomial} $Q_{\Tp}\in\calQ$ such that
\begin{equation}
\label{eq:tile:ball}
B_{I}(Q_{\Tp},0.2) \subset \calQ(\Tp) \subset B_{I}(Q_{\Tp},1),
\end{equation}
where $B_{I}(Q,r)$ denotes the ball with center $Q$ and radius $r$ with respect to the norm \eqref{eq:normI},
\item\label{def:tile:cover} for each grid cube $I \in \calD$, the uncertainty regions $\Set{\calQ(\Tp) \given \Tp\in\TP_{I}}$ form a disjoint cover of $\calQ$, and
\item\label{def:tile:nested} if $I \subseteq I'$, $\Tp \in \TP_{I}$, $\Tp' \in \TP_{I'}$, then either $\calQ(\Tp) \cap \calQ(\Tp') = \emptyset$ or $\calQ(\Tp) \supseteq \calQ(\Tp')$.
\end{enumerate}
\end{lemma}
This is similar to the construction of Christ grid cubes, but easier, because we can start at a smallest scale, and we do not need a small boundary property.

The requirement \eqref{eq:tile:ball} on the uncertainty regions $\calQ(\Tp)$ is dictated by Lemma~\ref{lem:osc-int}.
The uncertainty regions used in \cite{MR2545246,arXiv:1105.4504} in the case $\ds=1$ also satisfy \eqref{eq:tile:ball}, up to multiplicative constants.
However, it is convenient not to prescribe the exact shape of the uncertainty regions, in order to obtain the nestedness property \eqref{def:tile:nested}.
\begin{proof}
For each $I\in\calD$, choose a maximal $0.7$-separated subset $\calQ_{I} \subset \calQ$ with respect to the $I$-norm \eqref{eq:normI}.
Then the balls $B_{I}(Q,0.3)$, $Q\in \calQ_{I}$, are disjoint, and the balls $B_{I}(Q,0.7)$, $Q\in \calQ_{I}$, cover $\calQ$.
Hence, there exists a partition $\calQ = \cup_{Q\in \calQ_{I}} \tilde\calQ(I,Q)$ such that $B_{I}(Q,0.3) \subset \tilde\calQ(I,Q) \subset B_{I}(Q,0.7)$.
We fix such a partition for each $I \in \calD$.

Next, we construct nested partitions $\calQ = \cup_{Q\in \calQ_{I}} \calQ(I,Q)$ for $I \in \calD_{s}$ in decreasing order of the scale $s$, starting with $s=\sumax$.
For $I \in \calD_{\sumax}$ and $Q\in\calQ_{I}$, we let $\calQ(I,Q) := \tilde\calQ(I,Q)$.

Suppose now that $\calQ(I',Q)$ have been constructed for some $I'\in\calD$.
For a grid cube $I\in\calD$ with $\widehat{I} = I'$ and $Q\in \calQ_{I}$, we define
\[
\calQ(I,Q) := \cup_{Q' \in \calQ_{I'} \cap \tilde\calQ(I,Q)} \calQ(I',Q').
\]
By downward induction on $s$, we will show that, for every $I\in\calD_{s}$ and $Q \in \calQ_{I}$, we have
\begin{equation}
\label{eq:tile:ball-construction}
B_{I}(Q,0.2) \subset \calQ(I,Q) \subset B_{I}(Q,1).
\end{equation}
For $s=\sumax$, this holds by construction, so suppose that $s < \sumax$ and that the claim is known with $I$ replaced by $I':=\widehat{I}$.
For every $Q' \in \calQ_{I'} \setminus \tilde\calQ(I,Q)$ and $\tilde{Q} \in \calQ(I',Q')$, we have
\[
\norm{Q-\tilde{Q}}_{I}
\geq
\norm{Q-Q'}_{I}
-
\norm{Q'-\tilde{Q}}_{I}
\geq
0.3
-
10^{-4} \norm{Q'-\tilde{Q}}_{I'}
\geq
0.3
-
10^{-4} \cdot 1
\geq
0.2,
\]
where we used \eqref{eq:normQ:parent} and the inductive hypothesis.
This shows the first inclusion in \eqref{eq:tile:ball-construction}.

For every $Q' \in \calQ_{I'} \cap \tilde\calQ(I,Q)$ and $\tilde{Q} \in \calQ(I',Q')$, we have
\[
\norm{Q-\tilde{Q}}_{I}
\leq
\norm{Q-Q'}_{I}
+
\norm{Q'-\tilde{Q}}_{I}
\leq
0.7
+
10^{-4} \norm{Q'-\tilde{Q}}_{I'}
\leq
0.7
+
10^{-4} \cdot 1
\leq
1,
\]
where we again used \eqref{eq:normQ:parent} and the inductive hypothesis.
This shows the second inclusion in \eqref{eq:tile:ball-construction}.

Finally, the required collections of pairs will be defined by
\[
\TP_{I} := \Set{ (I,\calQ(I,Q)) \given Q \in \calQ_{I}}.
\]
We have verified \eqref{eq:tile:ball} in \eqref{eq:tile:ball-construction}, and the remaining properties easily follow from the construction.
\end{proof}

\begin{definition}
\label{def:tile}
Fixing a choice of collections of pairs from Lemma~\ref{lem:tile}, we write
\[
\TP := \bigcup_{s=\sumin}^{\sumax} \bigcup_{I\in\calD_{s}} \TP_{I}.
\]
The pairs in the set $\TP$ are called \emph{tiles}.
\end{definition}

For a pair $\Tp$, let
\begin{align*}
E(\Tp) &:= \Set{x\in I_{\Tp} \given Q_x\in \calQ(\Tp) \land \smin(x) \leq \scale(\Tp) \leq \smax(x)},\\
\mE(\Tp) &:= \Set{x\in I_{\Tp} \given Q_x\in \calQ(\Tp) \land \scale(\Tp) \leq \smax(x)}.
\end{align*}
The need for the latter notion becomes apparent in the tree estimate, see \eqref{eq:leaf-mass}.

For every tile $\Tp\in\TP$, we define the corresponding operator
\begin{equation}
\label{eq:Ttile}
T_{\Tp}f(x) := \one_{E(\Tp)}(x) \int e(Q_{x}(x)-Q_{x}(y)) \CZK_{\scale(\Tp)}(x,y) f(y) \dif y.
\end{equation}
\end{definition}
The tile operators and their adjoints
\begin{equation}
\label{eq:Ttile*}
T_{\Tp}^{*}g(y) = \int e(-Q_{x}(x)+Q_{x}(y)) \overline{\CZK_{\scale(\Tp)}(x,y)} (\one_{E(\Tp)}g)(x) \dif x
\end{equation}
have the support properties
\begin{equation}
\label{eq:Ttile-supp}
\supp T_{\Tp}f \subseteq I_{\Tp},
\qquad
\supp T_{\Tp}^{*}g \subseteq I_{\Tp}^{*} := 2 I_{\Tp}
\end{equation}
for any $f,g \in L^{1}_{\mathrm{loc}}(\R^{\ds})$.
For a collection of tiles $\TC \subset \TP$, we write $T_{\TC} := \sum_{\Tp\in\TC} T_{\Tp}$.
Then the linearized operator \eqref{eq:T-linearized} can be written as $T_{\TP}$.

\subsection{General notation}
The characteristic function of a set $I$, as well as the corresponding multiplication operator, is denoted by $\one_{I}$.
The \emph{Hardy--Littlewood maximal operator} is given by
\[
Mf(x) := \sup_{x\in I}\frac{1}{\abs{I}}\int_{I}\abs{f},
\]
the latter supremum being taken over all (not necessarily grid) cubes containing $x$.
For $1<q<\infty$, the \emph{$q$-maximal operator} is given by
\begin{equation}
\label{eq:HL-q-max-op}
M_{q}f := (M \abs{f}^{q})^{1/q}.
\end{equation}

Parameters $\epsilon,\eta$ (standing for small numbers) and $C$ (standing for large numbers) are allowed to change from line to line, but may only depend on $d,\ds,\tau$, and the implicit constants related to $\CZK$, unless an additional dependence is indicated by a subscript.

For $A,B>0$, we write $A\lesssim B$ (resp. $A\gtrsim B$) in place of $A<CB$ (resp. $A>CB$).
If the constant $C=C_{\delta}$ depends on some quantity $\delta$, then we may write $A\lesssim_{\delta}B$.

The operator norm on $L^{2}(\R^{\ds})$ is denoted by $\norm{T}_{2\to 2} := \sup_{\norm{f}_{2} \leq 1} \norm{Tf}_{2}$.

\section{Tree selection algorithm}
\label{sec:dens-selection}

\subsection{Spatial decomposition}
\label{sec:spatial-decomposition}
We begin with a simplified version of V.~Lie's stopping time construction from \cite{arXiv:1105.4504}.

\begin{definition}\label{def:ord}
Let $\Tp,\Tp'$ be pairs.
We say that
\begin{align*}
\Tp< \Tp' &:\iff I_{\Tp}\subsetneq I_{\Tp'} \text{ and } \calQ(\Tp') \subseteq \calQ(\Tp),\\
\Tp\leq \Tp' &:\iff I_{\Tp}\subseteq I_{\Tp'} \text{ and } \calQ(\Tp') \subseteq \calQ(\Tp).
\end{align*}
\end{definition}
The relations $<$ and $\leq$ are transitive, similarly to \cite{MR0340926} and differently from \cite{arXiv:1105.4504}.

\begin{definition}
A \emph{stopping collection} is a subset $\calF \subset\calD$ of the form $\calF=\cup_{k\geq 0}\calF_{k}$, where each $\calF_{k}$ is a collection of pairwise disjoint cubes such that, for each $F\in\calF_{k+1}$, there exists $F'\in\calF_{k}$ with $F'\supsetneq F$ ($F'$ is called the \emph{stopping parent of $F$}).
The collection of \emph{stopping children} of $F\in\calF_{k}$ is $\ch_{\calF}(F) := \Set{ F'\in\calF_{k+1} \given F'\subset F}$.
More generally, the collection of \emph{stopping children} of $I\in\calD$ is $\ch_{\calF}(I) := \Set{ F\in\calF \text{ maximal} \given F\subsetneq I }$.
We denote by $\ch^{m}$ the set of children of $m$-th generation, that is,
\[
\ch^{0}(I) := \Set{I},
\quad
\ch^{m+1}(I) := \cup_{I'\in\ch^{m}(I)} \ch(I').
\]
\end{definition}
\begin{lemma}
\label{lem:spatial-decomposition}
There exists a stopping collection $\calF$ with the following properties.
\begin{enumerate}
\item\label{it:F0}
$\calF_{0} = \calD_{\sumax}$.
\item\label{it:children}
For each $F\in\calF$, we have
\begin{equation}
\label{eq:Lie-support-decay}
\sum_{F'\in\ch(F)} \abs{F'} \leq D^{-10 \ds} \abs{F}.
\end{equation}
\item\label{it:stopping-neighbors}
For each $k \geq 0$, the set of grid cubes
\begin{equation}
\label{eq:tildeCk}
\tilde\calC_{k} := \Set{ I\in\calD \given \exists F\in\calF_{k} : I \subseteq F}
\end{equation}
satisfies
\begin{equation}
\label{eq:tildeCk-Whitney}
I \in \tilde\calC_{k}, I'\in\calD, I' \subset 3I, s(I') < s(I)
\implies
I' \in \tilde\calC_{k}.
\end{equation}
\item\label{it:counting-fct}
For $k\geq 0$, define the set of grid cubes
\begin{equation}
\label{eq:Ck}
\calC_{k} := \tilde\calC_{k} \setminus \tilde\calC_{k+1}
\end{equation}
and the corresponding set of tiles
\begin{equation}
\label{eq:Pnk}
\TP_{k} := \Set{ \Tp\in\TP \given I_{\Tp} \in \calC_{k}}.
\end{equation}
Then, for every $n\geq 1$, the set of tiles
\begin{equation}
\label{eq:Mnk}
\TM_{n,k} := \Set{ \Tp\in\TP_{k} \text{ maximal w.r.t.~``$<$'' } \given \abs{ \mE(\Tp)}/\abs{I_{\Tp}} \geq 2^{-n} }
\end{equation}
satisfies
\begin{equation}
\label{eq:Mnk-counting}
\norm[\big]{\sum_{\Tp\in\TM_{n,k}} \one_{I_{\Tp}}}_{\infty} \lesssim 2^{n} \log (n+1).
\end{equation}
\end{enumerate}
\end{lemma}
The stopping property \eqref{it:stopping-neighbors} can be informally stated by saying that each stopping cube is completely surrounded by stopping cubes of the same generation $k$ and similar (up to $\pm 1$) scale.
This is very useful for handling tail estimates.
\begin{proof}
We start with $\calF_{0} := \calD_{\sumax}$ being the set of all cubes of the maximal spatial scale, this is part \ref{it:F0} of the conclusion.
Part \ref{it:stopping-neighbors} holds with $k=0$, since $\tilde\calC_{0}=\calD$.

Let now $k\geq 0$ and suppose that $\calF_{k}$ has been constructed already.
Let
\begin{equation}
\label{eq:tildeMnk}
\tilde\TM_{n,k} := \Set{ \Tp\in\TP \text{ maximal w.r.t.~``$<$'' } \given \abs{ \mE(\Tp)}/\abs{I_{\Tp}} \geq 2^{-n} \text{ and } I_{\Tp}\in \tilde\calC_{k} }.
\end{equation}
Since the sets $\mE(\Tp)$ corresponding to $\Tp\in\tilde\TM_{n,k}$ are pairwise disjoint, we have the Carleson packing condition
\[
\sum_{\Tp\in\tilde\TM_{n,k} : I_{\Tp}\subseteq J} \abs{I_{\Tp}}
\leq
2^{n}
\sum_{\Tp\in\tilde\TM_{n,k} : I_{\Tp}\subseteq J} \abs{\mE(\Tp)}
\leq
2^{n} \abs{J}
\text{ for every }
J\in\calD.
\]
Let $C$ be a large constant to be chosen later, and for $F\in\calF_{k}$ let
\begin{equation}
\label{eq:B(F)}
B(F) :=
\bigcup_{n\geq 1}\Set[\Big]{\sum_{\Tp\in\tilde\TM_{n,k} : I_{\Tp}\subseteq F} \one_{I_{\Tp}} \geq C2^{n}\log(n+1) }.
\end{equation}
By the John--Nirenberg inequality, we obtain
\[
\abs{B(F)}
\lesssim
\sum_{n\geq 1} e^{-c \frac{C2^{n}\log(n+1)}{2^{n}}} \abs{F}
\lesssim
\Big(\sum_{n\geq 1} (n+1)^{-cC}\Big) \abs{F}.
\]
The numerical constant on the right-hand side can be made arbitrarily small by taking $C$ sufficiently large.
Let $\calJ(F) \subset \calD$ be the set of grid cubes contained in $B(F)$, and let $\calJ'(F) \subset \calD$ be the minimal set containing $\calJ(F)$ and satisfying the analog of \eqref{eq:tildeCk-Whitney}, namely
\[
I \in \calJ'(F), I'\in\calD, I' \subset 3I, s(I') < s(I)
\implies
I' \in \calJ'(F).
\]
Let $\calF_{k+1}$ be the set of maximal cubes in $\tilde\calC_{k+1} = \cup_{F\in\calF_{k}} \calJ'(F)$.
Then Part~\ref{it:stopping-neighbors} of the conclusion with $k$ replaced by $k+1$ holds by construction.
Part~\ref{it:counting-fct} of the conclusion holds for the given $k$, because $\TM_{n,k} \subseteq \tilde\TM_{n,k}$, and we removed all tiles whose spatial cubes are contained in the sets \eqref{eq:B(F)}, where the overlap is large.

Let us now verify part \ref{it:children} of the conclusion for $F\in\calF_{k}$.
By disjointness of the maximal cubes, we have
\begin{align*}
\sum_{F'\in\ch(F)} \abs{F'}
&\leq
\sum_{\tilde F \in \calF_{k}} \abs{F \cap \cup \calJ'(\tilde F)}\\
&\leq
\sum_{\tilde F \in \calF_{k} : F \cap \cup \calJ'(\tilde F) \neq \emptyset} \abs{\cup \calJ'(\tilde F)}\\
&\lesssim
\sum_{\tilde F \in \calF_{k} : F \cap \cup \calJ'(\tilde F) \neq \emptyset} \abs{B(\tilde F)}\\
&\lesssim
c \sum_{\tilde F \in \calF_{k} : F \cap \cup \calJ'(\tilde F) \neq \emptyset} \abs{\tilde F},
\end{align*}
where the constant $c>0$ can be made arbitrarily small by choosing a suitably large $C$.
Moreover, if $\tilde F \in \calF_{k}$ is such that $F \cap \cup \calJ'(\tilde F) \neq \emptyset$, then $\dist(\tilde F,F) \lesssim \sup_{F' \in \calJ(\tilde{F})} \ell(F')$, and, choosing $C$ sufficiently large, we may assume $\dist(\tilde F,F) \leq \ell(\tilde F)$.
It follows that $s(\tilde F) \leq s(F) + 1$, since otherwise $3 \tilde F \supset \hat{F}$ and $s(\tilde F) > s(\hat{F})$, so that $\hat{F} \in \tilde\calC_{k}$ by the inductive hypothesis \eqref{eq:tildeCk-Whitney}, contradicting $F\in\calF_{k}$.
Therefore, the sum over $\tilde F$ in the above display is $\lesssim \abs{F}$.
\end{proof}

\subsection{Forest selection}
\label{sec:tree-selection}

A set of tiles $\TA\subset\TP$ is called an \emph{antichain} if no two tiles in $\TA$ are related by ``$<$'' (this is the standard order theoretic term for a concept already used in \cite{MR0340926} under a different name).
A set of tiles $\TC\subset\TP$ is called \emph{convex} if
\[
\Tp_1,\Tp_2\in\TC, \Tp\in\TP, \Tp_1< \Tp < \Tp_2 \implies \Tp\in\TC.
\]
We call a subset $\TD\subset\TC$ of a convex set $\TC\subset\TP$ a \emph{down subset} if $\Tp < \Tp'$ with $\Tp\in\TC$ and $\Tp'\in\TD$ implies $\Tp\in\TD$.
Unions of down subsets are again down subsets.
Both down subsets and their relative complements are convex.

For $a\geq 1$ and a tile $\Tp$, we will write $a\Tp$ for the pair $(I_{\Tp},B_{I_{\Tp}}(Q_{\Tp},a))$.
Counterintuitively, for $a' \geq a \geq 1$ and a tile $\Tp$, we have $a'\Tp \leq a\Tp$; this notational inconsistency cannot be avoided without breaking the convention used in all time-frequency analysis literature starting with \cite{MR0340926}.

\begin{definition}\label{def:tree}
A \emph{tree} (of generation $k$) is a convex collection of tiles $\TT\subset\TP_{k}$ together with a \emph{top tile} $\Tp_0 = \top\TT \in \TP_{k}$ such that for all $\Tp\in\TT$ we have $4\Tp<\Tp_0$.
To each tree, we associate the \emph{central polynomial} $Q_{\TT} = Q_{\top\TT}$ and the spatial cube $I_{\TT} = I_{\top \TT}$.
\end{definition}

\begin{definition}
\label{def:dist}
For $\Tp\in\TP$ and $Q\in\calQ$, we write
\[
\Delta(\Tp,Q) := \norm{Q_{\Tp}-Q}_{I_{\Tp}} + 1.
\]
\end{definition}

\begin{definition}\label{def:sep}
Two trees $\TT_1$ and $\TT_2$ are called \emph{$\Delta$-separated} if
\begin{enumerate}
\item $\Tp\in\TT_1\ \land\ I_{\Tp}\subseteq I_{\TT_{2}} \implies \Delta(\Tp,Q_{\TT_2})>\Delta$ and
\item $\Tp\in\TT_2\ \land\ I_{\Tp}\subseteq I_{\TT_{1}} \implies \Delta(\Tp,Q_{\TT_1})>\Delta$.
\end{enumerate}
\end{definition}

\begin{remark}
If $I_{\TT_{1}}\cap I_{\TT_{2}}=\emptyset$, then $\TT_{1}$ and $\TT_{2}$ are $\Delta$-separated for any $\Delta$.
\end{remark}

\begin{definition}\label{def:F-forest}
Let $n,k\in\N$.
An \emph{$L^{\infty}$-forest of level $n$ and generation $k$} is a disjoint union $\TF=\cup_{j}\TT_j$ of $2^{C n}$-separated trees $\TT_{j} \subset \TP_{k}$ (with a large constant $C$ to be chosen later) such that
\begin{equation}\label{eq:F-forest-counting}
\norm[\big]{\sum_{j}\one_{I_{\TT_{j}}}}_{\infty} \lesssim 2^n \log(n+1).
\end{equation}
\end{definition}

\begin{definition}\label{def:mass}
We define the \emph{maximal density} of a tile $\Tp\in\TP$ by
\begin{equation}\label{eq:densk}
\mdens_{k}(\Tp):=\sup_{\lambda\geq 2} \lambda^{-\dim\calQ} \sup_{\Tp'\in\TP_{k} : \lambda\Tp \leq \lambda\Tp'} \frac{\abs{\mE(\lambda \Tp')}}{\abs{I_{\Tp'}}},
\end{equation}
where $\dim\calQ$ is the dimension of the vector space $\calQ$.
We also write $\mdens_{k}(\TS) = \sup_{\Tp\in\TS} \mdens_{k}(\Tp)$ for sets of tiles $\TS\subset\TP_{k}$.
The subset of ``heavy'' tiles is defined by
\begin{equation}
\label{eq:heavy}
\TH_{n,k} := \Set{ \Tp\in\TP_{k} \given \mdens_{k}(\Tp) > C_{0} 2^{-n}},
\end{equation}
where $C_{0} = C_{0}(d,\ds) > 1$ is a sufficiently large constant to be chosen later.
\end{definition}
The maximal density is monotonic in the sense that if $\Tp_{1}\leq\Tp_{2}$ are in $\TP_{k}$, then $\mdens_{k}(\Tp_{1}) \geq \mdens_{k}(\Tp_{2})$.
Indeed, in this case by \eqref{eq:normQ:parent} we have $\lambda\Tp_{1}\leq\lambda\Tp_{2}$ for every $\lambda\geq 2$, and the claim follows by transitivity of $\leq$.
It follows that each set $\TH_{n,k} \subset \TP_{k}$ is a down subset, and in particular convex.

\begin{proposition}
\label{prop:fef-forest}
For every $n \geq 1$ and every $k\geq 0$, the set $\TH_{n,k}$ can be represented as the disjoint union of $O(n^{2})$ antichains and $O(n)$ $L^{\infty}$-forests of level $n$ and generation $k$.
\end{proposition}
\begin{proof}
We would like to avoid the $\lambda$-dilates in Definition~\ref{def:mass}.
To this end, we consider the down subset of $\TP_{k}$
\[
\TC_{n,k} := \Set{ \Tp\in\TP_{k} \given \exists \Tm\in\TM_{n,k} : 2\Tp < 100\Tm }.
\]
We claim that the remaining set of tiles $\TH_{n,k}\setminus\TC_{n,k}$ can be partitioned into at most $n$ antichains.
Indeed, otherwise there exists a chain $\Tp_{0}<\dotsb<\Tp_{n}$ inside $\TH_{n,k}\setminus\TC_{n,k}$.
By definition \eqref{eq:densk}, there exists $\lambda\geq 2$ and a tile $\Tp'\in\TP_{k}$ such that $\lambda\Tp_{n} \leq \lambda\Tp'$ and
\begin{equation}
\label{eq:largeElambdaTp'}
\abs{\mE(\lambda \Tp')}/\abs{I_{\Tp'}} > C_{0} 2^{-n} \lambda^{\dim\calQ}.
\end{equation}
By the John ellipsoid theorem applied to the unit ball of the norm $\norm{\cdot}_{I_{\Tp'}}$, the set $\calQ(\lambda \Tp')$ can be covered by $O(\lambda^{\dim\calQ})$ uncertainty regions of the form $\calQ(\Tp'')$, where $\Tp'' \in \TP_{k}$ are tiles with $I_{\Tp''}=I_{\Tp'}$ and $\norm{Q_{\Tp'}-Q_{\Tp''}}_{I_{\Tp'}} \leq \lambda+1$.
For at least one such tile, we have $\abs{\mE(\Tp'')} \gtrsim C_{0} 2^{-n} \abs{I_{\Tp''}}$, so that $\abs{\mE(\Tp'')} > 2^{-n} \abs{I_{\Tp''}}$ provided that $C_{0}$ in \eqref{eq:heavy} is sufficiently large.
By definition \eqref{eq:Mnk}, there exists $\Tm \in \TM_{n,k}$ with $\Tp'' \leq \Tm$.

From \eqref{eq:largeElambdaTp'}, we obtain
\[
\lambda \leq \lambda^{\dim\calQ} < 2^{n} \abs{\mE(\lambda \Tp')}/\abs{I_{\Tp'}} \leq 2^{n},
\]
and it follows from \eqref{eq:normQ:parent} that, for all $Q\in \calQ(100\Tm)$, we have
\begin{align*}
\norm{Q_{\Tp_{0}}-Q}_{I_{\Tp_{0}}}
&\leq
\norm{Q_{\Tp_{0}}-Q_{\Tp_{n}}}_{I_{\Tp_{0}}}
+
\norm{Q_{\Tp_{n}}-Q_{\Tp'}}_{I_{\Tp_{0}}}
+
\norm{Q_{\Tp'}-Q_{\Tp''}}_{I_{\Tp_{0}}}\\
&\quad+
\norm{Q_{\Tp''}-Q_{\Tm}}_{I_{\Tp_{0}}}
+
\norm{Q_{\Tm}-Q}_{I_{\Tp_{0}}}\\
&\leq
1
+
10^{-4n}( \norm{Q_{\Tp_{n}}-Q_{\Tp'}}_{I_{\Tp_{n}}}
+
\norm{Q_{\Tp'}-Q_{\Tp''}}_{I_{\Tp'}}\\
&\quad+
\norm{Q_{\Tp''}-Q_{\Tm}}_{I_{\Tp''}}
+
\norm{Q_{\Tm}-Q}_{I_{\Tm}} )\\
&\leq
1 + 10^{-4n}( \lambda + (\lambda+1) + 1 + 100 )
\leq
2.
\end{align*}
Hence $2 \Tp_{0} \leq 100 \Tm$, contradicting the choice $\Tp_{0}\not\in\TC_{n,k}$.

We want to show that $\TC_{n,k}$ can be decomposed into $O(n)$ $L^{\infty}$-forests and $O(n^{2})$ antichains; then since $\TH_{n,k}$ is convex the same will hold for $\TH_{n,k}\cap\TC_{n,k}$.
Let
\[
\TB(\Tp):=\Set{ \Tm \in \TM_{n,k} \given 100 \Tp\leq \Tm},
\qquad \Tp\in\TC_{n,k}.
\]
In view of \eqref{eq:Mnk-counting}, we have $1\leq \abs{B(\Tp)} \lesssim 2^{n} \log(n+1)$ for every $\Tp\in\TC_{n,k}$.
Let
\[
\TC_{n,k,j}:=
\Set{ \Tp\in\TC_{n,k} \given 2^j\leq \abs{\TB(\Tp)} < 2^{j+1} }.
\]
For the remaining part of the proof, fix $j\geq 0$ such that $2^{j} \lesssim 2^{n} \log(n+1)$.
It suffices to show that $\TC_{n,k,j}$ can be written as the union of an $L^{\infty}$-forest and $O(n)$ antichains.

First we verify that the set $\TC_{n,k,j}$ is convex.
Indeed, if $\Tp_{1}<\Tp<\Tp_{2}$ with $\Tp_{1},\Tp_{2}\in\TC_{n,k,j}$ and $\Tp\in\TC_{n,k}$, then $100 \Tp_{1} < 100 \Tp < 100 \Tp_{2}$, so that $\TB(\Tp_{1}) \supseteq \TB(\Tp) \supseteq \TB(\Tp_{2})$, so that $\Tp\in\TC_{n,k,j}$.

Let $\TU\subseteq\TC_{n,k,j}$ be the set of tiles $\Tu$ such that there is no $\Tp\in\TC_{n,k,j}$ with $I_{\Tu} \subsetneq I_{\Tp}$ and $\calQ(100 \Tu) \cap \calQ(100 \Tp) \neq\emptyset$.
These are our candidates for being tree tops.

In order to verify the counting function estimate \eqref{eq:F-forest-counting}, we will show that for every $x\in\R^{\ds}$ the set $\TU(x) := \Set{ \Tu\in\TU \given x\in I_{\Tu}}$ has cardinality $O(2^{-j}2^{n}\log (n+1))$.
The family $\TU(x)$ can be subdivided into $O(1)$ families, denoted by $\TU'(x)$, in each of which the sets $\calQ(100\Tu)$, $\Tu\in\TU'(x)$, are disjoint (just make this decomposition at each scale independently).
In particular, the sets $\TB(\Tu)$, $\Tu\in\TU'(x)$, are pairwise disjoint.
These sets have cardinality at least $2^{j}$, and their union has cardinality at most $2^{n} \log(n+1)$ by \eqref{eq:Mnk-counting}.
This implies $\abs{\TU'(x)} \lesssim 2^{-j} 2^{n} \log(n+1)$.

Let
\[
\TD(\Tu) := \Set{\Tp\in \TC_{n,k,j} \given 2 \Tp < \Tu},
\quad
\Tu\in\TU.
\]
We will show that
\[
\TA'_{j} := \TC_{n,k,j} \setminus \cup_{\Tu\in\TU} \TD(\Tu)
\]
is an antichain.
Suppose that, on the contrary, there exist $\Tp, \Tp_1\in \TA'_{j}$ with $\Tp< \Tp_1$.
We claim that, in this case, for every $l=1,2,\dotsc$, there exists a sequence of tiles $\Tp_{1},\dotsc,\Tp_{l}\in \TC_{n,k,j}$ with
\[
2\Tp < 200\Tp_{1} < \dotsb < 200\Tp_{l}.
\]
This will produce a contradiction, because the spatial cubes of these tiles are in $\calC_{k}$, and therefore have bounded scale.
For $l=1$, the claim follows from \eqref{eq:normQ:parent}.
Suppose now that the claim is known for some $l\geq 1$.
If $\Tp_{l} \in \TU$, then $\Tp\in\TD(\Tp_{l})$, and this is a contradiction.
Otherwise, by definition of $\TU$, there exists a tile $\Tp_{l+1}\in\TC_{n,k,j}$ such that $I_{\Tp_{l}} \subsetneq I_{\Tp_{l+1}}$ and $\calQ(100 \Tp_{l}) \cap \calQ(100 \Tp_{l+1}) \neq \emptyset$.
It follows from \eqref{eq:normQ:parent} that $\calQ(200 \Tp_{l}) \supseteq \calQ(200 \Tp_{l+1})$, hence $200 \Tp_{l} < 200 \Tp_{l+1}$.
This finishes the proof of the claim and of the fact that $\TA'_{j}$ is an antichain.

Let $\TU' := \Set{ \Tu\in\TU \given \TD(\Tu) \neq\emptyset}$ and introduce on this set the relation
\begin{equation}
\label{eq:propto:def}
\Tu \propto \Tu'
:\iff
\exists \Tp\in \TD(\Tu) \text{ with } 10 \Tp\leq \Tu'.
\end{equation}
We claim that
\begin{equation}
\label{eq:propto-Fef-trick}
\Tu\propto\Tu' \implies I_{\Tu} = I_{\Tu'} \text{ and } \calQ(100 \Tu)\cap \calQ(100 \Tu')\neq\emptyset.
\end{equation}
\begin{proof}[Proof of the claim \eqref{eq:propto-Fef-trick}]
Let $\Tu,\Tu'\in\TU'$ with $\Tu\propto \Tu'$.
By definition, there exists $\Tp\in \TC_{n,k,j}$ with $2\Tp < \Tu$ and $10 \Tp\leq \Tu'$.

First we notice that it suffices to show that
\begin{equation}
\label{eq:propto-Fef-trick:100}
\calQ(100 \Tu)\cap \calQ(100 \Tu')\neq\emptyset.
\end{equation}
Indeed, the spatial cubes $I_{\Tu},I_{\Tu'}$ both contain $I_{\Tp}$, so, unless they coincide, they are strictly nested, contradicting $\Tu,\Tu'\in\TU$.

Now we make a case distinction.
If $I_{\Tp} = I_{\Tu'}$, then $100 \Tu' \leq 2\Tp < \Tu$, and \eqref{eq:propto-Fef-trick:100} follows.

In the case $I_{\Tp} \subsetneq I_{\Tu'}$, we deduce from \eqref{eq:normQ:parent} that $100 \Tp < 100 \Tu'$ and $100 \Tp < 100 \Tu$.
If \eqref{eq:propto-Fef-trick:100} does not hold, then the sets $\TB(\Tu)$ and $\TB(\Tu')$ are disjoint.
On the other hand, $\TB(\Tp) \supseteq \TB(\Tu) \cup \TB(\Tu')$, so that $\abs{\TB(\Tp)} \geq \abs{\TB(\Tu)} + \abs{\TB(\Tu')} \geq 2\cdot 2^{j}$, a contradiction to $\Tp\in\TC_{n,k,j}$.
This establishes \eqref{eq:propto-Fef-trick:100}.%
\footnote{This counting argument is due to C.~Fefferman \cite[p.~569]{MR0340926}.}
\end{proof}

Next, we verify that ``$\propto$'' is an equivalence relation.
Let $\Tu,\Tu',\Tu'' \in \TU'$ be such that $I_{\Tu}=I_{\Tu'}=I_{\Tu''}$, $\calQ(100\Tu) \cap \calQ(100\Tu') \neq \emptyset$, and $\calQ(100\Tu')\cap\calQ(100\Tu'') \neq \emptyset$.
For all, and since $\TD(\Tu)\neq\emptyset$ in particular for some, $\Tp\in \TD(\Tu)$ we have $2 \Tp < \Tu$.
By \eqref{eq:normQ:parent} this implies $4 \Tp < 1000 \Tu$, and it follows that
\begin{equation}
\label{eq:propto-equiv-rel}
4\Tp < \Tu''.
\end{equation}
Using \eqref{eq:propto-Fef-trick} and the fact that \eqref{eq:propto-equiv-rel} implies $\Tu \propto \Tu''$, we deduce transitivity, symmetry, and reflexivity of the relation ``$\propto$''.

Let $\TV\subseteq\TU'$ be a set of representatives for equivalence classes modulo $\propto$, and let
\[
\TT(\Tv) := \cup_{\Tu \propto \Tv} \TD(\Tu),
\quad
\Tv\in\TV.
\]
Each $\TT(\Tv)$ is a union of down subsets $\TD(\Tu) \subset \TC_{n,k,j}$, and therefore convex.
It follows from \eqref{eq:propto-equiv-rel} that each $\TT(\Tv)$ is a tree with top $\Tv$.
It follows from \eqref{eq:propto:def} that these trees satisfy the separation condition
\begin{equation}\label{eq:10sep}
\forall \Tv\neq \Tv' \quad \forall \Tp\in\TT(\Tv) \qquad 10\Tp\not\leq \Tv'.
\end{equation}

In order to upgrade the condition \eqref{eq:10sep} to $2^{C n}$-separateness, it suffices to remove the bottom $O(n)$ layers of tiles.%
\footnote{In order to perform this step, C.~Fefferman used tiles with ``central'' frequency intervals, see \cite[Section 5]{MR0340926}.
In order to avoid this restriction and the associated averaging argument, V.~Lie has introduced a separation condition similar to \eqref{eq:10sep}, see \cite[Proposition 2, hypothesis 2]{MR2545246}.}
More precisely, for $l=1,\dotsc,Cn$, let $\TA_{n,k,j,l}$ be the set of minimal tiles in $\cup_{\Tv\in\TV} \TT(\Tv) \setminus \cup_{l'<l} \TA_{n,k,j,l'}$.
Then each $\TA_{n,k,j,l}$ is an antichain, and the total number of antichains $\TA_{n,k,j,l}$ for all $j,l$ is $O(n^{2})$.
Each $\TT'(\Tv) := \TT(\Tv) \setminus \cup_{l} \TA_{n,k,j,l}$ is still a convex set, hence a tree with top $\Tv$.
Moreover, it follows from \eqref{eq:10sep} that tiles in distinct trees $\TT(\Tv)$ are not comparable.
Therefore, for every $\Tp\in\TT'(\Tv)$, there exist tiles $\Tp_{1}<\dotsb<\Tp_{Cn}<\Tp$ in $\TT(\Tv)$.
If $I_{\Tp} \subseteq I_{\Tv'}$ for some $\Tv'\neq\Tv$, then, using \eqref{eq:normQ:parent} and \eqref{eq:10sep} for the tile $\Tp_{1}$, we obtain
\[
\norm{Q_{\Tp}-Q_{\Tv'}}_{I_{\Tp}}
\geq
(10^{4})^{Cn} \norm{Q_{\Tp}-Q_{\Tv'}}_{I_{\Tp_{1}}}
\geq
(10^{4})^{Cn} \cdot 9,
\]
and this implies $10^{4Cn}$-separateness.
\end{proof}

The trees supplied by Proposition~\ref{prop:fef-forest} at different levels $n$ need not be disjoint.
We will now make them disjoint.
Let $\TT_{n,k,j,l}'$ be the trees and $\TA_{n,k,j}'$ the antichains provided by Proposition~\ref{prop:fef-forest} at level $n\geq 1$ and generation $k$.
For $n=1$, define
\[
\TT_{n,k,j,l} := \TT_{n,k,j,l}',
\quad
\TA_{n,k,j} := \TA_{n,k,j}'.
\]
For $n>1$, define
\[
\TT_{n,k,j,l} := \TT_{n,k,j,l}' \setminus \TH_{n-1,k},
\quad
\TA_{n,k,j} := \TA_{n,k,j}' \setminus \TH_{n-1,k}.
\]
Since we remove down subsets, the sets $\TT_{n,k,j,l}$ are still (convex) trees.

These sets have the following properties.
\begin{enumerate}
\item The set of all tiles can be decomposed as the disjoint union
\begin{equation}\label{eq:tree-dec}
\TP = \bigcup_{n=1}^{\infty} \bigcup_{k\in\N} \big( \bigcup_{j\lesssim n} \bigcup_{l} \TT_{n,k,j,l} \cup \bigcup_{j\lesssim n^{2}} \TA_{n,k,j}\big).
\end{equation}
\item Each $\TA_{n,k,j}$ is an antichain.
\item\label{it:TT-convex} Each $\TT_{n,k,j,l}$ is a tree.
\item Each $\TF_{n,k,j} := \cup_{l} \TT_{n,k,j,l}$ is an $L^{\infty}$-forest of level $n$ and generation $k$.
\item $\mdens_{k}(\TF_{n,k,j}) \lesssim 2^{-n}$.
\item $\mdens_{k}(\TA_{n,k,j}) \lesssim 2^{-n}$.
\end{enumerate}

\section{Estimates for error terms}
In this section, we consider error terms coming from antichains and boundary parts of trees.
These terms are morally easier to handle than the main terms, in the sense that they are controlled by positive operators (after a suitable $TT^{*}$ argument).

\subsection{The basic \texorpdfstring{$TT^{*}$}{TT*} argument}
\begin{lemma}\label{lem:sep-tiles}
Let $\Tp_1, \Tp_2\in\TP$ with $\meas{I_{\Tp_{1}}} \leq \meas{I_{\Tp_{2}}}$.
Then
\begin{equation}\label{eq:sep-tiles}
\abs[\Big]{\int T_{\Tp_1}^{*}g_{1} \overline{T_{\Tp_2}^{*}g_{2}}}
\lesssim
\frac{\Delta(\Tp_1,Q_{\Tp_2})^{-\frac{\tau}{d}}}{\abs{I_{\Tp_{2}}}}
\int_{E(\Tp_1)}\abs{g_{1}}\int_{E(\Tp_2)}\abs{g_{2}}.
\end{equation}
\end{lemma}

\begin{proof}
We may assume $I_{\Tp_{1}}^{*} \cap I_{\Tp_{2}}^{*} \neq \emptyset$, since otherwise the left-hand side of the conclusion vanishes.
Expanding the left-hand side of \eqref{eq:sep-tiles}, we obtain
\begin{multline*}
\abs[\Big]{\int
\int e(-Q_{x_{1}}(x_{1})+Q_{x_{1}}(y)) \overline{\CZK_{\scale(\Tp_{1})}(x_{1},y)} (\one_{E(\Tp_{1})}g_{1})(x_{1}) \dif x_{1}\\
\cdot \overline{\int e(-Q_{x_{2}}(x_{2})+Q_{x_{2}}(y)) \overline{\CZK_{\scale(\Tp_{2})}(x_{2},y)} (\one_{E(\Tp_{2})}g_{2})(x_{2}) \dif x_{2}} \dif y}\\
\leq
\int_{E(\Tp_{1})} \int_{E(\Tp_{2})} \abs[\Big]{\int
  e((Q_{x_{1}}-Q_{x_{2}})(y)-Q_{x_{1}}(x_{1})+Q_{x_{2}}(x_{2}))\\
  \cdot \overline{\CZK_{\scale(\Tp_{1})}(x_{1},y)} \CZK_{\scale(\Tp_{2})}(x_{2},y) \dif y}
\abs{g_{1}(x_{1}) g_{2}(x_{2})} \dif x_{2} \dif x_{1}.
\end{multline*}
By Lemma~\ref{lem:osc-int} applied to the cube $I_{\Tp_{1}}^{*}$, the integral inside the absolute value is bounded by
\[
(\norm{Q_{x_{1}}-Q_{x_{2}}}_{I_{\Tp_{1}}^{*}}+1)^{-\tau/d}/\abs{I_{\Tp_{2}}},
\]
and the conclusion follows, since
\begin{align*}
\norm{Q_{x_{1}}-Q_{x_{2}}}_{I_{\Tp_{1}}^{*}}
&\geq
\norm{Q_{\Tp_{1}}-Q_{\Tp_{2}}}_{I_{\Tp_{1}}^{*}}
-
\norm{Q_{\Tp_{1}}-Q_{x_{1}}}_{I_{\Tp_{1}}^{*}}
-
\norm{Q_{\Tp_{2}}-Q_{x_{2}}}_{I_{\Tp_{1}}^{*}}\\
&\geq
\norm{Q_{\Tp_{1}}-Q_{\Tp_{2}}}_{I_{\Tp_{1}}}
-
C\norm{Q_{\Tp_{1}}-Q_{x_{1}}}_{I_{\Tp_{1}}}
-
\norm{Q_{\Tp_{2}}-Q_{x_{2}}}_{C I_{\Tp_{2}}}\\
&\geq
\Delta(\Tp_{1},Q_{\Tp_{2}})-1-C
-
C \norm{Q_{\Tp_{2}}-Q_{x_{2}}}_{I_{\Tp_{2}}}\\
&\geq
\Delta(\Tp_{1},Q_{\Tp_{2}})-C.
\qedhere
\end{align*}
\end{proof}

\subsection{Antichains and boundary parts of trees}
A separate treatment of boundary parts of trees was introduced in \cite{arXiv:1105.4504} and allows to preserve the sharp spatial support of adjoint tree operators $T_{\TT}^{*}$ throughout the main argument in Section~\ref{sec:main-est}, while avoiding exceptional sets in \cite{MR0340926}.
\begin{lemma}
\label{lem:antichain-supp}
There exists $\epsilon=\epsilon(d,\ds)>0$ such that, for every $0\leq\eta\leq 1$, every $1\leq\rho\leq\infty$, every antichain $\TA\subseteq \TP_{k}$, and every $Q\in\calQ$, we have
\begin{equation}\label{eq:antichain-supp}
\norm{ \sum_{\Tp\in \TA} \Delta(\Tp,Q)^{-\eta} \one_{E(\Tp)} }_{\rho}
\lesssim
\mdens_{k}(\TA)^{\epsilon \eta/\rho} \abs[\big]{\cup_{\Tp\in\TA} I_{\Tp}}^{1/\rho}.
\end{equation}
\end{lemma}
\begin{proof}
Since the sets $E(\Tp)$, $\Tp\in\TA$, are disjoint, the claimed estimate clearly holds for $\rho=\infty$.
Hence, by H\"older's inequality, it suffices to consider $\rho=1$.
Let also $\delta=\mdens_{k}(\TA)$.
We have to show
\[
\sum_{\Tp\in\TA} \Delta(\Tp,Q)^{-\eta} \abs{E(\Tp)}
\lesssim
\delta^{\eta\epsilon} \abs{S},
\quad
S=\cup_{\Tp\in\TA} I_{\Tp}.
\]
Let $\epsilon>0$ be a small number to be chosen later and split the summation in two parts.
For those $\Tp\in\TA$ with $\Delta(\Tp,Q) \geq \delta^{-\epsilon}$, the estimate clearly holds, because the sets $E(\Tp)$ are pairwise disjoint.

Let $\TA' = \Set{\Tp\in\TA \given \Delta(\Tp,Q) < \delta^{-\epsilon}}$, and consider the collection $\calL$ of the maximal grid cubes $L\in\calD$ such that $L\subsetneq I_{\Tp}$ for some $\Tp\in\TA'$ and $I_{\Tp} \not\subseteq L$ for all $\Tp\in\TA'$.
The collection $\calL$ is a disjoint cover of the set $\cup_{\Tp\in\TA'} I_{\Tp}$.
Fix $L\in\calL$; we will show that
\[
\sum_{\Tp\in\TA'} \abs{E(\Tp) \cap L}
\lesssim
\delta^{1-\epsilon \dim\calQ} \abs{L}.
\]
The conclusion will follow with $\epsilon = 1/(\dim\calQ+1)$.

By construction, $\hat{L} \in \calC_{k}$ and there exists a tile $\Tp_{L}\in\TA'$ with $I_{\Tp_{L}}\subseteq\hat{L}$.
If $I_{\Tp_{L}}=\hat{L}$, let $\Tp_{L}':=\Tp_{L}$, otherwise let $\Tp_{L}'$ be the unique tile with $I_{\Tp_{L}'} = \hat{L}$ and $Q\in\calQ(\Tp_{L}')$.
In both cases, with $\lambda = C \delta^{-\epsilon}$ for a sufficiently large constant $C$, the tile $\Tp_{L}'$ satisfies
\begin{enumerate}
\item $\lambda\Tp_{L} \leq \lambda\Tp_{L}'$ and
\item for every $\Tp\in\TA'$ with $L\cap I_{\Tp} \neq \emptyset$, we have $\lambda\Tp_{L}' \leq \Tp$.
\end{enumerate}
In view of disjointness of $E(\Tp)$'s, this implies
\[
\sum_{\Tp\in\TA'} \abs{E(\Tp) \cap L}
\leq
\abs{\mE(\lambda\Tp_{L}')}
\leq
\lambda^{\dim\calQ} \abs{I_{\Tp_{L}'}} \mdens_{k}(\Tp_{L})
\lesssim
\delta^{1-\epsilon \dim\calQ} \abs{L}.
\qedhere
\]
\end{proof}

For a tree $\TT$, the \emph{boundary component} is defined by
\begin{equation}
\label{eq:T-boundary}
\bd(\TT) := \Set{ \Tp\in\TT \given I_{\Tp}^{*} \not\subseteq I_{\TT}}.
\end{equation}
Notice that $\bd(\TT)$ is an up-set: if $\Tp\in\bd(\TT)$, $\Tp'\in\TT$, $\Tp \leq \Tp'$, then $I_{\Tp'}^{*} \supseteq I_{\Tp}^{*}$, so that also $\Tp'\in\bd(\TT)$.
In particular, $\TT \setminus \bd(\TT)$ is still a (convex) tree.

\begin{proposition}\label{prop:sf}
Fix $n,j$, and let either $\TS = \cup_{k}\cup_{l} \bd(\TT_{n,k,j,l})$ or $\TS = \cup_{k}\TA_{n,k,j}$.
Then
\begin{equation}
\label{eq:sf}
\norm{T_{\TS}}_{2 \to 2}
\lesssim
2^{-\epsilon n}.
\end{equation}
\end{proposition}

\begin{proof}
For $\Tp\in\TP$, let $\gen(\Tp)$ denote the unique natural number such that $\Tp \in \TP_{\gen(\Tp)}$.
For $\Tp'\in\TS$, let
\[
\TD(\Tp') := \Set{ \Tp\in\TS \given
  \scale(\Tp) \leq \scale(\Tp') \land
  \gen(\Tp) \geq \gen(\Tp') \land
  I_{\Tp}^{*}\cap I_{\Tp'}^{*}\neq\emptyset}.
\]
Then $I_{\Tp} \subset 5 I_{\Tp'}$ for $\Tp\in\TD(\Tp')$.

We claim that, for every $\Tp,\Tp' \in \TS$ with $I_{\Tp}^{*}\cap I_{\Tp'}^{*}\neq\emptyset$, at least one of the relations $\Tp\in\TD(\Tp')$ or $\Tp'\in\TD(\Tp)$ holds.
Indeed, otherwise we may assume $\scale(\Tp') < \scale(\Tp)$ and $\gen(\Tp') < \gen(\Tp)$.
Then $I_{\Tp'} \subset 3 I_{\Tp}$, and, since $I_{\Tp} \in \tilde\calC_{\gen(\Tp)}$, it follows from \eqref{eq:tildeCk-Whitney} that $I_{\Tp'} \in \tilde\calC_{\gen(\Tp)}$.
But then $\gen(\Tp') \geq \gen(\Tp)$, a contradiction.

Using the above claim and Lemma~\ref{lem:sep-tiles}, we obtain
\begin{align*}
\int \abs[\big]{ T_{\TS}^{*}g }^2
&\leq
2 \sum_{\Tp'\in\TS} \sum_{\Tp\in\TD(\Tp')}
\abs[\Big]{\int T_{\Tp'}^{*}g \overline{T^{*}_{\Tp}g} }\\
&\lesssim
\sum_{\Tp'\in\TS}\int_{E(\Tp')}\abs{g} \sum_{\Tp\in \TD(\Tp')} \Delta(\Tp,Q_{\Tp'})^{-\tau/d} \frac{\int_{E(\Tp)}\abs{g}}{\abs{I_{\Tp'}}}.
\end{align*}
By H\"older's inequality with exponent $1<q<2$, this is
\begin{equation}
\label{eq:Vf}
\leq
\sum_{\Tp'\in\TS}\int_{E(\Tp')} \abs{g} \left(\frac{\int_{5I_{\Tp'}} \abs{g}^q}{\abs{I_{\Tp'}}}\right)^{\frac{1}{q}}
\frac{\norm{ \sum_{\Tp\in \TD(\Tp')} \Delta(\Tp,Q_{\Tp'})^{-\tau/d} \one_{E(\Tp)} }_{q'}}{\abs{I_{\Tp'}}^{\frac{1}{q'}}}.
\end{equation}

First we will show that the last fraction is $O(2^{-\epsilon n})$ uniformly in $\Tp'$.
Let $k' := \gen(\Tp')$, so that $\TD(\Tp') \subset \cup_{k\geq k'} \TP_{k}$.

We begin by estimating the spatial support of $\TD(\Tp') \cap \TP_{k}$.
If $F\in\calF_{k'+1}$ and $F\cap 5I_{\Tp'} \neq \emptyset$, then $s(F) \leq s(\Tp')$, since otherwise an ancestor of $I_{\Tp'}$ would have been included in $\calF_{k'+1}$ by part \eqref{it:stopping-neighbors} of Lemma~\ref{lem:spatial-decomposition}.
Therefore, by \eqref{eq:Lie-support-decay} for $k>k'$, we have
\begin{multline}
\label{eq:Dp'-support}
\meas[\big]{\bigcup_{F\in\calF_{k}} F \cap 5 I_{\Tp'}}
\lesssim
\meas[\big]{\bigcup_{F\in\calF_{k} : F\cap 5 I_{\Tp'} \neq \emptyset} F}\\
\lesssim
e^{k'-k}
\meas[\big]{\bigcup_{F\in\calF_{k'+1} : F\cap 5 I_{\Tp'} \neq \emptyset} F}
\lesssim
e^{k'-k} \meas{I_{\Tp'}},
\end{multline}
and the same estimate also clearly holds for $k=k'$.

Next, we decompose $\TD(\Tp')$ into antichains.
Consider first the case $\TS = \cup_{k,l} \bd(\TT_{n,k,j,l})$.
For $k\geq k'$ and $m\geq 0$, let
\[
\TA_{k,m} := \bigcup_{l} \Set{ \Tp \in \TD(\Tp') \cap \bd(\TT_{n,k,j,l}) \given \scale(\Tp) = s(k,l) - m },
\]
where
\[
s(k,l) :=
\begin{cases}
\min(s(\top \TT_{n,k',j,l}), s(\Tp')) & \text{if } k=k',\\
s(\top \TT_{n,k,j,l}) & \text{if } k>k'.
\end{cases}
\]
The sets $\TA_{k,m}$ are pairwise disjoint antichains and partition $\TD(\Tp') = \cup_{k\geq k', m\geq 0}\TA_{k,m}$.
We have
\begin{align*}
\abs[\Big]{\bigcup_{\Tp\in \TA_{k',m}} I_{\Tp}}
&\leq
\sum_{l} \meas[\Big]{5 I_{\Tp'} \cap \bigcup_{\substack{\Tp\in \bd(\TT_{n,k',j,l}) :\\ s(\Tp) = s(k',l) - m}} I_{\Tp}}\\
&\leq
\sum_{l} \meas[\Big]{5 I_{\Tp'} \cap \Set{x\in I_{\TT_{n,k',j,l}} \given \dist(x,\R^{\ds}\setminus I_{\TT_{n,k',j,l}}) < C D^{s(k',l) - m}} }\\
&\lesssim
D^{-m} \sum_{l} \meas[\Big]{5 I_{\Tp'} \cap I_{\TT_{n,k',j,l}}}\\
&\lesssim
D^{-m} 2^{n} \log(n+1) \meas{I_{\Tp'}},
\end{align*}
where we have used \eqref{eq:F-forest-counting} in the last step.
Analogously, using \eqref{eq:Dp'-support} for $k>k'$, we obtain
\begin{align*}
\abs[\Big]{\bigcup_{\Tp\in \TA_{k,m}} I_{\Tp}}
&\leq
\sum_{l : I_{\TT_{n,k,j,l}} \cap 5I_{\Tp'} \neq \emptyset} \meas[\Big]{\bigcup_{\substack{\Tp\in \bd(\TT_{n,k,j,l}) :\\ s(\Tp) = s(k,l) - m}} I_{\Tp}}\\
&\leq
\sum_{l : I_{\TT_{n,k,j,l}} \cap 5I_{\Tp'} \neq \emptyset}
\meas[\Big]{\Set{x\in I_{\TT_{n,k,j,l}} \given \dist(x,\R^{\ds}\setminus I_{\TT_{n,k,j,l}}) < C D^{s(k,l) - m}} }\\
&\lesssim
\sum_{l : I_{\TT_{n,k,j,l}} \cap 5I_{\Tp'} \neq \emptyset}
D^{-m} \meas[\Big]{I_{\TT_{n,k,j,l}}}\\
&\lesssim
D^{-m} \sum_{F\in\calF_{k} : F\cap 5 I_{\Tp'} \neq \emptyset} \meas{F} 2^{n} \log(n+1)\\
&\lesssim
e^{k'-k} D^{-m} 2^{n} \log(n+1) \meas{I_{\Tp'}}.
\end{align*}
Combining this with a trivial estimate coming from \eqref{eq:Dp'-support}, we obtain
\begin{equation}
\label{eq:Akm-support}
\abs[\Big]{\bigcup_{\Tp\in \TA_{k,m}} I_{\Tp}} \lesssim e^{k'-k} \min(1,C2^{n}\log(n+1) D^{-m}) \abs{I_{\Tp'}}.
\end{equation}
In the case $\TS = \cup_{k} \TA_{n,k,j}$, we define $\TA_{k,0} := \TA_{n,k,j} \cap \TD(\Tp')$ and $\TA_{k,m} := \emptyset$ for $m>0$.
The estimate \eqref{eq:Akm-support} also holds in this case.

Using Lemma~\ref{lem:antichain-supp} with $\rho=q'$ and $0\leq\eta\leq 1$ and \eqref{eq:Akm-support}, it follows that
\begin{multline}\label{eq:sparse-Lie-tree-Carl-Delta}
\frac{\norm{ \sum_{\Tp\in \TD(\Tp')} \Delta(\Tp,Q_{\Tp'})^{-\eta} \one_{E(\Tp)} }_{\rho}}{\abs{I_{\Tp'}}^{1/\rho}}
\leq
\sum_{k\geq k', m\geq 0} \frac{\norm{ \sum_{\Tp\in \TA_{k,m}} \Delta(\Tp,Q_{\Tp'})^{-\eta} \one_{E(\Tp)} }_{\rho}}{\abs{I_{\Tp'}}^{1/\rho}}\\
\lesssim
2^{-\epsilon\eta n/\rho} \sum_{k\geq k', m\geq 0} \frac{\abs{\cup_{\Tp\in \TA_{k,m}} I_{\Tp}}^{1/\rho}}{\abs{I_{\Tp'}}^{1/\rho}}\\
\lesssim
2^{-\epsilon\eta n/\rho} \sum_{k\geq k', m\geq 0} e^{(k'-k)/\rho} \min(1,C2^{n}\log(n+1) D^{-m})^{1/\rho}\\
\lesssim_{\rho}
2^{-\epsilon\eta n/\rho} n \sum_{k\geq k'} e^{(k'-k) / \rho}
\lesssim_{\rho}
2^{-\epsilon\eta n/\rho} n.
\end{multline}
Using the estimate \eqref{eq:sparse-Lie-tree-Carl-Delta} with $\eta=\tau/d$ in the last factor of \eqref{eq:Vf}, we obtain the claimed exponential decay in $n$.

In order to conclude, it now suffices to show
\[
\sum_{\Tp\in\TS} \int_{E(\Tp)} \abs{g} (g)_{5I_{\Tp},q}
\lesssim
n \norm{g}_2^{2},
\text{ where }
(g)_{5I,q} := (\abs{I}^{-1} \int_{5I} \abs{g}^{q} )^{1/q}.
\]
Similarly to the estimate \eqref{eq:sparse-Lie-tree-Carl-Delta} with $\eta=0$, we obtain the Carleson packing condition
\begin{equation}
\label{eq:sparse-Lie-tree-Carl}
\norm{ \sum_{\Tp\in\TS : I_{\Tp}\subseteq J}\one_{E(\Tp)} }_{\rho} \lesssim_{\rho} n \abs{J}^{1/\rho},
\quad 1\leq \rho<\infty.
\end{equation}
Let $\calS\subset\calD$ be the stopping time associated to the average $(g)_{5I,q}$, that is, $\ch_{\calS}(I)$ are the maximal cubes $J \subset I$ with $(g)_{5J,q} > C (g)_{5I,q}$ for some large constant $C$.
Since the $q$-maximal operator \eqref{eq:HL-q-max-op} has weak type $(q,q)$, the family $\calS$ is \emph{sparse} in the sense that there exist pairwise disjoint subsets $\calE(I) \subseteq I\in\calS$ with $\abs{\calE(I)} \gtrsim \abs{I}$ (one can take $\calE(I) = I \setminus \cup_{J\in\ch_{\calS}(I)} J$).
Then
\begin{align*}
\sum_{\Tp\in\TS} \int_{E(\Tp)} \abs{g}  (g)_{5I_{\Tp},q}
&\lesssim
\sum_{I\in\calS} (g)_{5I,q} \int \abs{g} \sum_{\Tp\in\TS, I_{\Tp}\subseteq I} \one_{E(\Tp)}\\
\text{by H\"older}&\leq
\sum_{I\in\calS} (g)_{5I,q} \abs{I} (g)_{I,q} \Big( \abs{I}^{-1} \int_{I} \big( \sum_{\Tp\in\TS, I_{\Tp}\subseteq I} \one_{E(\Tp)} \big)^{q'} \Big)^{1/q'}\\
\text{by \eqref{eq:sparse-Lie-tree-Carl} and sparseness}&\lesssim
n \sum_{I\in\calS} (g)_{5I,q} \abs{\calE(I)} (g)_{I,q}\\
\text{by disjointness}&\lesssim
n \int (M_{q}g)^{2}
\lesssim
n \norm{g}_{2}^{2},
\end{align*}
where we have used the strong type $(2,2)$ inequality for $M_{q}$, $q<2$, in the last step.
\end{proof}

\subsection{Localization}
In order to handle exponents $p\neq 2$, we localize the operator $T_{\TS}$.
\begin{proposition}
\label{prop:sf-loc}
Let $\TS$ be as in Proposition~\ref{prop:sf}.
Let $F,G \subseteq \R^{\ds}$ be such that
\begin{equation}
\label{eq:sf-loc:assume}
\meas{I_{\Tp} \cap G} \lesssim \nu \meas{I_{\Tp}}
\text{ and }
\meas{5I_{\Tp} \cap F} \lesssim \kappa \meas{I_{\Tp}}
\text{ for every }
\Tp\in\TS.
\end{equation}
Then, for every $0\leq \alpha < 1/2$, we have
\begin{equation}
\label{eq:sf-loc}
\norm{\one_{G} T_{\TS} \one_{F}}_{2 \to 2}
\lesssim_{\alpha}
\nu^{\alpha} \kappa^{\alpha} 2^{-\epsilon n}.
\end{equation}
\end{proposition}
\begin{proof}
Taking a geometric average with \eqref{eq:sf}, it suffices to show
\[
\norm{\one_{G} T_{\TS} \one_{F}}_{2 \to 2}
\lesssim
n \nu^{\alpha} \kappa^{\alpha}.
\]
To this end, we replace \eqref{eq:antichain-supp} by the estimate
\[
\norm{ \sum_{\Tp\in \TA} \one_{E(\Tp) \cap G} }_{\rho}
\lesssim
\abs[\big]{\cup_{\Tp\in\TA} I_{\Tp} \cap G}^{1/\rho}
\lesssim
\nu^{1/\rho} \abs[\big]{\cup_{\Tp\in\TA} I_{\Tp}}^{1/\rho}
\]
for all antichains $\TA\subset\TS$.
Following the proof of the Carleson packing condition \eqref{eq:sparse-Lie-tree-Carl}, we obtain
\begin{equation}
\label{eq:sparse-Lie-tree-Carl-loc}
\norm{ \sum_{\Tp\in\TS : I_{\Tp}\subseteq J}\one_{E(\Tp) \cap G} }_{\rho}
\lesssim_{\rho}
n \nu^{1/\rho} \abs{J}^{1/\rho},
\quad 1\leq \rho<\infty.
\end{equation}
Fix functions $f,g$ with $\supp f \subset F$ and $\supp g\subset G$.
Consider the stopping time $\calS \subset \Set{I_{\Tp} \given \Tp\in\TS }$ associated to the average $(f)_{5I,1}$ and let $\calE(I) \subset I\in\calS$ be pairwise disjoint subsets with $\abs{\calE(I)} \gtrsim \abs{I}$.
With $\alpha=1/q'$, we obtain
\begin{align*}
\int \abs{g T_{\TS} f}
&\lesssim
\sum_{\Tp\in\TS} (f)_{5 I_{\Tp}, 1} \int_{E(\Tp)} \abs{g}\\
&\lesssim
\sum_{I\in\calS} (f)_{5I,1} \int \sum_{\Tp\in\TS, I_{\Tp}\subset I} \one_{E(\Tp)} \abs{g}\\
&\lesssim
\sum_{I\in\calS} (f)_{5I,q} (\one_{F})_{5I,q'} \meas{I} (g)_{I,q} \Big( \meas{I}^{-1} \int \big(\sum_{\Tp\in\TS, I_{\Tp}\subset I} \one_{E(\Tp) \cap G} \big)^{q'} \Big)^{1/q'}\\
&\lesssim
n \kappa^{1/q'} \nu^{1/q'} \sum_{I\in\calS} (f)_{5I,q} \meas{\calE(I)} (g)_{I,q}\\
&\lesssim
n \kappa^{1/q'} \nu^{1/q'} \int (M_{q}f) (M_{q}g)\\
&\lesssim
n \kappa^{1/q'} \nu^{1/q'} \norm{M_{q}f}_{2} \norm{M_{q}g}_{2}\\
&\lesssim
n \kappa^{1/q'} \nu^{1/q'} \norm{f}_{2} \norm{g}_{2}.
\qedhere
\end{align*}
\end{proof}

\section{Estimates for trees and forests}
\label{sec:main-est}
In this section, we consider the bulk of tiles that are organized into trees.
The contribution of each tree will be estimated by a maximally truncated operator associated to the kernel $\CZK$.

\subsection{Cotlar's inequality}
We call a subset $\sigma\subset\Z$ \emph{convex} if it is order convex, that is, $s_{1}<s<s_{2}$ and $s_{1},s_{2}\in\sigma$ implies $s\in\sigma$.
For a measurable function $\sigma$ that maps $\R^{\ds}$ to the set of finite convex subsets of $\Z$, we consider the associated truncated singular integral operator
\begin{equation}
\label{eq:T-trunc}
T_{\sigma}f(x) := \sum_{s\in\sigma(x)} \int \CZK_{s}(x,y) f(y) \dif y.
\end{equation}
An inspection of the proof of Cotlar's inequality, see e.g.~\cite[Section I.7.3]{MR1232192}, shows that the non-tangentially maximally truncated operator
\begin{equation}
\label{eq:nontang-max-trunc}
T_{\calN}f (x) := \sup_{\sigma}\sup_{\abs{x-x'} \leq C D^{\min\sigma(x)}} \abs{T_{\sigma}f(x')}
\end{equation}
is bounded on $L^{p}(\R^{\ds})$, $1<p<\infty$ (more precisely, the proof of Cotlar's inequality in the above reference shows that this holds if the constant $C$ is sufficiently small, see also \cite[Lemma 3.2]{MR3484688}; one can subsequently pass to larger values of $C$, see e.g.~\cite[Section II.2.5.1]{MR1232192}).
We refer to this fact as the non-tangential Cotlar inequality.

We will use truncated singular integral operators with sets of scales given by trees.
\begin{definition}
\label{def:T-scales}
For a tree $\TT$, we define
\begin{align*}
\sigma (\TT,x) &:=
\Set{\scale(\Tp) \given \Tp\in\TT, x\in E(\Tp) },\\
\smax (\TT, x) &:= \max \sigma(x),\\
\smin (\TT,x) &:= \min \sigma(x).
\end{align*}
\end{definition}
We will omit the argument $\TT$ if it is clear from the context.
By construction of the set of all tiles $\TP$, the set $\sigma (\TT,x)$ is convex in $\Z$ for every tree $\TT$ and every $x\in\R^{\ds}$.

\subsection{Tree estimate}
\begin{definition}
\label{def:leaves}
For a non-empty finite collection of tiles $\TS\subset\TP$,
\begin{enumerate}
\item let $\calJ(\TS) \subset \calD$ be the collection of the maximal grid cubes $J$ such that $100 D J$ does not contain $I_{\Tp}$ for any $\Tp\in\TS$, and
\item let $\calL(\TS) \subset \calD$ be the collection of the maximal grid cubes $L$ such that $L\subsetneq I_{\Tp}$ for some $\Tp\in\TS$ and $I_{\Tp} \not\subseteq L$ for all $\Tp\in\TS$.
\end{enumerate}

For a collection of pairwise disjoint grid cubes $\calJ\subset\calD$, we define the projection operator
\begin{equation}
\label{eq:positive-tree-proj}
P_{\calJ}f := \sum_{J\in\calJ} \one_{J} \meas{J}^{-1} \int_{J} f.
\end{equation}
\end{definition}
For later use, we note that the scales of adjacent cubes in $\calJ(\TS)$ differ at most by $1$, in the sense that if $J,J'\in\calJ$ and $\dist(J,J')\leq 10\max(\ell(J),\ell(J'))$, then $\abs{s(J)-s(J')} \leq 1$.
Indeed, if $J,J'\in\calJ$, $s(J) \leq s(J')-2$, and $\dist(J,J') \leq 10 \ell(J')$, then $100D\hat{J} \subset 100DJ'$ does not contain any $I_{\Tp}$, $\Tp\in\TS$, contradicting maximality of $J$.
\begin{lemma}[Tree estimate]\label{lem:tree}
Let $\TT\subseteq\TP$ be a tree, $\calJ := \calJ(\TT)$, and $\calL := \calL(\TT)$.
Then, for every $1 < p < \infty$, $f\in L^{p}(\R^{\ds})$, and $g\in L^{p'}(\R^{\ds})$, we have
\begin{equation}\label{eq:tree:proj}
\abs[\Big]{\int_{\R^{\ds}} g T_{\TT} f}
\lesssim
\norm{P_{\calJ} \abs{f}}_{p} \norm{P_{\calL} \abs{g}}_{p'}.
\end{equation}
\end{lemma}
\begin{proof}
The conclusion \eqref{eq:tree:proj} will follow from the estimate
\begin{equation}
\label{eq:tree-2loc}
\sup_{x\in L} \abs{\overline{e(Q_{\TT})} T_{\TT} e(Q_{\TT}) f}(x)
\leq
C \inf_{x\in L} (M+S) P_{\calJ} \abs{f} (x)
+
\inf_{x\in L} \abs{T_{\calN} P_{\calJ} f (x)},
\end{equation}
where
\begin{enumerate}
\item $L\in\calL$ is arbitrary,
\item $Q_{\TT}$ denotes the central polynomial of $\TT$ (notice that the left-hand side is well-defined in the sense that it does not depend on the choice of the constant term of $Q_{\TT}$),
\item the operator $S$, while depending on $\TT$, is bounded on $L^{p}(\R^{\ds})$ for $1<p<\infty$ with constants independent of $\TT$, and
\item the non-tangentially maximally truncated singular integral $T_{\calN}$, defined in \eqref{eq:nontang-max-trunc}, is bounded on $L^{p}(\R^{\ds})$ by Cotlar's inequality.
\end{enumerate}

Let $\sigma=\sigma(\TT)$ be as in Definition~\ref{def:T-scales} and fix $x\in L\in\calL$.
By definition,
\begin{multline*}
\abs{\overline{e(Q_{\TT})} T_{\TT} e(Q_{\TT}) f(x)}\\
=
\abs[\Big]{ \sum_{s \in \sigma(x)}\int e(-Q_{\TT}(x)+Q_x(x)-Q_{x}(y)+Q_{\TT}(y))\CZK_s(x,y) f (y) \dif y }\\
\leq
\sum_{s \in \sigma(x)} \int \abs{e(Q_{\TT}(y)-Q_{x}(y)-Q_{\TT}(x)+Q_{x}(x))-1} \abs{\CZK_s(x,y)} \abs{f(y)} \dif y\\
+\abs[\big]{ T_{\sigma} P_{\calJ} f(x) }
+\abs[\big]{ T_{\sigma} (1-P_{\calJ}) f(x) }
=: A(x)+B(x)+C(x).
\end{multline*}
The term $B(x)$ is a truncated singular integral, and is dominated by $\inf_{L} T_{\calN} P_{\calJ} f$.

We turn to $A(x)$.
If $\CZK_{s}(x,y)\neq 0$, then $\abs{x-y}\lesssim D^{s}$, and in this case
\begin{multline*}
\abs{e(Q_{\TT}(y)-Q_{x}(y)-Q_{\TT}(x)+Q_{x}(x))-1}\\
\leq
\norm{Q_{x}-Q_{\TT}}_{B(x,CD^{s})}
\lesssim
D^{s-\smax(x)} \norm{Q_{x}-Q_{\TT}}_{B(x,CD^{\smax (x)})}
\lesssim
D^{s-\smax(x)},
\end{multline*}
where we have used Lemma~\ref{lem:normQ}.
For $x\in L\in\calL$, we have $\scale(L) \leq \smin(x)-1$, and it follows that
\[
A(x)
\lesssim
D^{-\smax(x)} \sum_{s\in\sigma(x)} D^{s(1-\ds)} \int_{B(x,0.5 D^{s})} \abs{f}(y) \dif y.
\]
Since the collection $\calJ$ is a partition of $\R^{\ds}$, this can be estimated by
\[
A(x)
\lesssim
D^{-\smax(x)} \sum_{s\in\sigma(x)} D^{s(1-\ds)} \sum_{J\in\calJ : J \cap B(x,0.5 D^{s}) \neq \emptyset} \int_{J} \abs{f}(y) \dif y.
\]
The expression on the right hand side does not change upon replacing $\abs{f}$ by $P_{\calJ}\abs{f}$.
Moreover,
\begin{equation}
\label{eq:calJ-smallness}
I_{\Tp}^{*} \cap J \neq \emptyset \text{ with } \Tp\in\TT \text{ and } J\in\calJ \implies J \subset 3I_{\Tp}.
\end{equation}
Hence, the sum over $J\in\calJ$ is in fact restricted to cubes contained in $B(x,CD^{s})$, so that
\begin{multline*}
A(x)
\lesssim
D^{-\smax(x)} \sum_{s\in\sigma(x)} D^{s(1-\ds)} \int_{B(x,CD^{s})} P_{\calJ} \abs{f}(y) \dif y\\
\lesssim
D^{-\smax(x)} \sum_{s\in\sigma(x)} D^{s} \inf_{L} M P_{\calJ} \abs{f}
\lesssim
\inf_{L} M P_{\calJ} \abs{f}.
\end{multline*}

It remains to treat $C(x)$.
Using \eqref{eq:calJ-smallness}, we estimate
\begin{multline*}
\abs{T_{\sigma} (1-P_{\calJ}) f(x)}
=
\abs[\big]{\sum_{\Tp\in\TT} \one_{E(\Tp)}(x) \int \CZK_{s(\Tp)}(x,y) ((1-P_{\calJ}) f)(y) \dif y}\\
\leq
\sum_{\Tp\in\TT} \one_{E(\Tp)}(x) \sum_{J\in\calJ : J \subseteq 3 I_{\Tp}} \sup_{y,y'\in J} \abs{\CZK_{s(\Tp)}(x,y)-\CZK_{s(\Tp)}(x,y')} \int_{J} \abs{f}\\
\lesssim
\sum_{I\in\calH} \one_{I}(x) \sum_{J\in\calJ : J \subseteq 3 I}
D^{-(\ds+\tau) \scale(I)} \diam(J)^{\tau} \int_{J} P_{\calJ} \abs{f},
\end{multline*}
where $\calH = \Set{ I_{\Tp} \given \Tp\in\TT}$.
The right-hand side of this inequality is constant on each $L\in\calL$.
Hence, we obtain \eqref{eq:tree-2loc} with
\[
S f(x) := \sum_{I\in\calD} \one_{I}(x) \sum_{J\in\calJ : J \subseteq 3 I}
D^{-(\ds+\tau) \scale(I)} \diam(J)^{\tau} \int_{J} f.
\]
It remains to obtain an $L^{p}$ estimate for the operator $S$.
We have
\begin{align*}
\abs[\big]{\int g S f}
&\lesssim
\sum_{I\in\calD,J\in\calJ : J \subseteq 3 I} (g)_{I} D^{\tau (\scale(J)-\scale(I))} \int_{J} \abs{f}\\
&\lesssim
\sum_{J\in\calJ} \int_{J} \abs{f} Mg \sum_{I\in\calD : J \subseteq 3 I} D^{\tau (\scale(J)-\scale(I))}\\
&\lesssim
\sum_{J\in\calJ} \int_{J} \abs{f} Mg\\
&\leq
\norm{f}_{p} \norm{Mg}_{p'}.
\end{align*}
By the Hardy--Littlewood maximal inequality and duality, it follows that $\norm{S}_{p\to p} \lesssim 1$ for $1<p<\infty$.
\end{proof}

\begin{corollary}
\label{cor:tree}
Let $\TT\subseteq\TP_{k}$ be a tree.
Let also $F\subseteq\R^{\ds}$ and $\kappa>0$ be such that
\begin{equation}
\label{eq:tree-loc}
I_{\Tp} \not\subseteq \Set{ M\one_{F} > \kappa }
\text{ for all }
\Tp\in\TT.
\end{equation}
Then, for every $1 < p < \infty$ and $f\in L^{p}(\R^{\ds})$, we have
\begin{equation}\label{eq:tree:local}
\norm{T_{\TT} \one_{F} f}_p
\lesssim
\kappa^{1/p'} \mdens_{k}(\TT)^{1/p} \norm{f}_{p}.
\end{equation}
\end{corollary}
Notice that the hypothesis \eqref{eq:tree-loc} holds with $\kappa=1$ and $F=\R^{\ds}$ for every tree $\TT$.
\begin{proof}
Fix $L\in\calL:=\calL(\TT)$.
By construction, $\hat{L} \in \calC_{k}$, and there exists a tile $\Tp_{L}\in\TT$ with $I_{\Tp_{L}}\subseteq\hat{L}$.
If $I_{\Tp_{L}}=\hat{L}$, let $\Tp_{L}':=\Tp_{L}$, otherwise let $\Tp_{L}'\in\TP_{k}$ be the unique tile with $I_{\Tp_{L}'} = \hat{L}$ and $Q_{\TT}\in\calQ(\Tp_{L}')$.
In both cases, the tile $\Tp_{L}'$ satisfies
\begin{enumerate}
\item $10\Tp_{L} \leq 10\Tp_{L}'$ and
\item for every $\Tp\in\TT$ with $L\cap I_{\Tp} \neq \emptyset$ we have $10\Tp_{L}' \leq \Tp$.
\end{enumerate}
It follows that the spatial support
\[
E(L) :=
L\cap \bigcup_{\Tp\in\TT} E(\Tp)
=
L\cap \bigcup_{\Tp\in\TT : I_{\Tp} \supset L} E(\Tp)
\]
satisfies
\begin{equation}
\label{eq:leaf-mass}
\abs{E(L)}
\leq
\abs{\mE(10\Tp_{L}')}
\leq
10^{\dim\calQ} \abs{I_{\Tp_{L}'}} \mdens_{k}(\Tp_{L})
\lesssim
\mdens_{k}(\TT) \abs{L}.
\end{equation}

It also follows from the hypothesis \eqref{eq:tree-loc} that
\begin{equation}
\label{eq:dens-kappa}
\meas{F \cap J} \lesssim \kappa \meas{J}
\end{equation}
for all $J\in\calJ:=\calJ(\TT)$.
Using Lemma~\ref{lem:tree}, H\"older's inequality, and the estimates \eqref{eq:leaf-mass} and \eqref{eq:dens-kappa}, we obtain
\begin{align*}
\abs[\Big]{\int_{\R^{\ds}} g T_{\TT} \one_{F} f}
&=
\abs[\Big]{\int_{\R^{\ds}} \sum_{L\in\calL}\one_{E(L)} g T_{\TT} \one_{F} f}\\
&\lesssim
\norm{P_{\calL} \abs[\big]{\sum_{L\in\calL}\one_{E(L)} g} }_{p'} \norm{P_{\calJ} \abs{\one_{F} f} }_{p}\\
&=
\Bigl( \sum_{L\in\calL} \meas{L} \bigl( \meas{L}^{-1} \int_{L} \one_{E(L)} \abs{g} \bigr)^{p'} \Bigr)^{1/p'}
\Bigl( \sum_{J\in\calJ} \meas{J} \bigl( \meas{J}^{-1} \int_{J} \one_{F} \abs{f} \bigr)^{p} \Bigr)^{1/p}\\
&\leq
\Bigl( \sum_{L\in\calL} \meas{L} \bigl( \meas{L}^{-1} \int_{L} \abs{g}^{p'} \bigr)
\bigl( \meas{L}^{-1} \int_{L} \one_{E(L)}^{p} \bigr)^{p'/p} \Bigr)^{1/p'}\\
&\quad \cdot
\Bigl( \sum_{J\in\calJ} \meas{J} \bigl( \meas{J}^{-1} \int_{J} \abs{f}^{p} \bigr)
\bigl( \meas{J}^{-1} \int_{J} \one_{F}^{p'} \bigr)^{p/p'} \Bigr)^{1/p}\\
&\lesssim
\mdens_{k}(\TT)^{1/p} \kappa^{1/p'}
\Bigl( \sum_{L\in\calL} \int_{L} \abs{g}^{p'} \Bigr)^{1/p'}
\Bigl( \sum_{J\in\calJ} \int_{J} \abs{f}^{p} \Bigr)^{1/p}\\
&\leq
\mdens_{k}(\TT)^{1/p} \kappa^{1/p'}
\norm{g}_{p'} \norm{f}_{p}.
\qedhere
\end{align*}
\end{proof}

\subsection{Separated trees}
\begin{definition}
\label{def:normal}
A tree $\TT$ is called \emph{normal} if for every $\Tp\in\TT$ we have $I_{\Tp}^{*} \subset I_{\TT}$.
\end{definition}
For a normal tree $\TT$, we have $\supp T_{\TT}^{*}g \subseteq I_{\TT}$ for every function $g$.
\begin{lemma}\label{lem:sep-trees}
There exists $\epsilon=\epsilon(d,\tau)>0$ such that, for any two $\Delta$-separated normal trees $\TT_1, \TT_{2}$, we have
\begin{equation}
\label{eq:sep-trees}
\abs[\Big]{\int_{\R^{\ds}} T_{\TT_{1}}^* g_{1} \overline{ T_{\TT_{2}}^* g_{2}} }
\lesssim
\Delta^{-\epsilon} \prod_{j=1,2} \norm{\abs{T_{\TT_{j}}^{*} g_{j}} + Mg_{j}}_{L^{2}(I_{\TT_{1}} \cap I_{\TT_{2}})}.
\end{equation}
\end{lemma}

\begin{proof}
The estimate clearly holds without decay in $\Delta$, so it suffices to consider $\Delta \gg 1$.
Without loss of generality, assume $I_{0} := I_{\TT_{1}} \subseteq I_{\TT_{2}}$ and $\TT_{1} \neq \emptyset$.

Recall that $Q_{\TT}$ denotes the central polynomial of a tree $\TT$ and let $Q := Q_{\TT_{1}}-Q_{\TT_{2}}$.
Let $0<\eta<1$ be chosen later, and let $\TS := \Set{\Tp \in \TT_{1}\cup\TT_{2} \given \norm{Q}_{I_{\Tp}} \geq \Delta^{1-\eta}}$.
It follows from the definition of $\Delta$-separation that
\[
\Tp \in (\TT_{1}\cup\TT_{2}) \land I_{\Tp} \subseteq I_{0}
\implies
\norm{Q}_{I_{\Tp}} \geq \Delta - 5,
\]
and the same still holds in the case $I_{\Tp} \supset I_{0}$ by monotonicity of the norms \eqref{eq:normI}.
Therefore, for sufficiently large $\Delta$, we may assume
\begin{equation}
\label{eq:sep-inside}
\Tp \in (\TT_{1}\cup\TT_{2})\setminus \TS
\implies
I_{\Tp} \cap I_{0} = \emptyset,
\end{equation}
and in partiular $\TT_{1} \subset \TS$.

Let $\calJ := \Set{ J \in \calJ(\TS)) \given J \subseteq I_{0}}$.
This is a partition of $I_{0}$.
Since the scales of adjacent cubes in this partition differ at most by $1$, there exists an adapted partition of unity $\one_{I_{0}} = \sum_{J\in\calJ} \chi_{J}$, where each $\chi_{J} : I_{0} \to [0,1]$ is a smooth function supported on $(1+1/D) J$ with $\abs{\nabla\chi_{J}} \lesssim \ell(J)^{-1}$.
We extend each $\chi_{J}$ to be zero on $\R^{\ds}\setminus I_{0}$; it will not matter that these extended functions are not necessarily continuous.

We claim that
\begin{equation}
\label{eq:leaf-separation}
\Delta_{J} := \norm{Q}_{J} \gtrsim \Delta^{1-\eta}
\text{ for all } J\in\calJ.
\end{equation}
Indeed, by definition there exists $\Tp \in \TS$ with $100D\hat{J} \supseteq I_{\Tp}$, and, by Lemma~\ref{lem:normQ}, we obtain
\[
\norm{Q}_{J}
\gtrsim
\norm{Q}_{100D\hat{J}}
\geq
\norm{Q}_{I_{\Tp}}
\gtrsim
\Delta^{1-\eta}.
\]

In order to prepare the application of Lemma~\ref{lem:osc-int}, we need to estimate local moduli of continuity of $T_{\TT}^{*}g$ for a tree $\TT$.
For every $\Tp\in\TT$ and $y,y'\in I_{\Tp}^{*}$, using \eqref{eq:Ks-size}, \eqref{eq:Ks-reg}, and Lemma~\ref{lem:normQ}, we obtain
\begin{align*}
\MoveEqLeft
\abs[\big]{e(Q_{\TT}(0)-Q_{\TT}(y))T_{\Tp}^{*}g(y) - e(Q_{\TT}(0)-Q_{\TT}(y'))T_{\Tp}^{*}g(y')}\\
&=
\abs[\Big]{\int \bigl(e(-Q_{x}(x)+Q_{x}(y)-Q_{\TT}(y)+Q_{\TT}(0)) \overline{\CZK_{\scale(\Tp)}(x,y)}\\
  & \quad
  - e(-Q_{x}(x)+Q_{x}(y')-Q_{\TT}(y')+Q_{\TT}(0)) \overline{\CZK_{\scale(\Tp)}(x,y')}\bigr) (\one_{E(\Tp)}g)(x) \dif x}\\
&\leq
\int_{E(\Tp)} \abs{g}(x) \abs[\Big]{e(-Q_{x}(y')+Q_{x}(y)-Q_{\TT}(y)+Q_{\TT}(y')) \overline{\CZK_{\scale(\Tp)}(x,y)} - \overline{\CZK_{\scale(\Tp)}(x,y')}}  \dif x\\
&\leq
\int_{E(\Tp)} \abs{g}(x) \Bigr( \abs{e(-Q_{x}(y')+Q_{x}(y)-Q_{\TT}(y)+Q_{\TT}(y'))-1} \abs{\CZK_{\scale(\Tp)}(x,y)}\\
&\quad
+ \abs{\overline{\CZK_{\scale(\Tp)}(x,y)} - \overline{\CZK_{\scale(\Tp)}(x,y')}} \Bigr)  \dif x\\
&\lesssim
\int_{E(\Tp)} \abs{g}(x) \Bigr( \norm{Q_{x}-Q_{\TT}}_{I_{\Tp}^{*}} \frac{\abs{y-y'}}{D^{s(\Tp)}} D^{-s(\Tp)\ds} + D^{-s(\Tp)\ds} \bigl(\frac{\abs{y-y'}}{D^{s(\Tp)}}\bigr)^{\tau} \Bigr)  \dif x\\
&\lesssim
\bigl(\frac{\abs{y-y'}}{D^{s(\Tp)}}\bigr)^{\tau} D^{-s(\Tp)\ds}\int_{E(\Tp)} \abs{g}(x) \dif x.
\end{align*}
Let $J\in\calD$ be such that for every $\Tp\in\TT$ we have $I_{\Tp}^{*} \cap (1+1/D) J \neq \emptyset \implies s(\Tp) \geq s(J)$.
Then, for every $y,y' \in (1+1/D) J$, we obtain
\begin{equation}
\label{eq:tree*-Hoelder}
\begin{split}
\MoveEqLeft
\abs[\big]{e(Q_{\TT}(0)-Q_{\TT}(y))T_{\TT}^{*}g(y) - e(Q_{\TT}(0)-Q_{\TT}(y'))T_{\TT}^{*}g(y')}\\
&\leq
\sum_{\Tp\in\TT : I_{\Tp}^{*} \cap (1+1/D) J \neq \emptyset}
\abs[\big]{e(Q_{\TT}(0)-Q_{\TT}(y))T_{\Tp}^{*}g(y) - e(Q_{\TT}(0)-Q_{\TT}(y'))T_{\Tp}^{*}g(y')}\\
&\lesssim
\sum_{s\geq s(J)} \sum_{\Tp\in\TT : I_{\Tp}^{*} \cap (1+1/D) J \neq \emptyset, s(\Tp)=s}
\bigl(\frac{\abs{y-y'}}{D^{s}}\bigr)^{\tau} D^{-s \ds}\int_{E(\Tp)} \abs{g}(x) \dif x\\
&\lesssim
\sum_{s\geq s(J)} \sum_{\Tp\in\TT : I_{\Tp}^{*} \cap (1+1/D) J \neq \emptyset, s(\Tp)=s}
\bigl(\frac{\abs{y-y'}}{D^{s}}\bigr)^{\tau} \inf_{J} Mg\\
&\lesssim
\bigl(\frac{\abs{y-y'}}{D^{s(J)}}\bigr)^{\tau} \inf_{J} Mg.
\end{split}
\end{equation}
The estimate \eqref{eq:tree*-Hoelder} implies in particular
\begin{equation}
\label{eq:tree*-loc-max}
\sup_{y\in (1+1/D) J} \abs{e(Q_{\TT}(0)-Q_{\TT}(y)) T_{\TT}^{*} g(y)}
\leq
\inf_{y\in \frac12 J} \abs{T_{\TT}^{*}g(y)}
+
C \inf_{y\in J} Mg(y).
\end{equation}
We claim that, for an absolute constant $s_{0}$, we have
\begin{equation}
\label{eq:separated-smalltiles}
\Tp\in\TT_{2}\setminus\TS, J\in\calJ, I_{\Tp}^{*} \cap J \neq \emptyset
\implies
s(\Tp) \leq s(J) + s_{0}.
\end{equation}
\begin{proof}[Proof of Claim \eqref{eq:separated-smalltiles}]
Let $s_{0} > 1$ be chosen later and suppose $s(\Tp)>s(J)+s_{0}$.
By definition, there exists $\Tp'\in\TS$ with $I_{\Tp'}\subseteq 100D\hat{J}$.
On the other hand, $10 I_{\Tp} \supset D^{s_{0}-1} \hat{J}$.
Using Lemma~\ref{lem:normQ}, we obtain
\[
\Delta^{1-\eta}
>
\norm{Q}_{I_{\Tp}}
\gtrsim
\norm{Q}_{10 I_{\Tp}}
\geq
\norm{Q}_{D^{s_{0}} \hat{J}}
\gtrsim
D^{s_{0}} \norm{Q}_{100D \hat{J}}
\geq
D^{s_{0}} \norm{Q}_{I_{\Tp}'}
\geq
D^{s_{0}} \Delta^{1-\eta},
\]
so that $1 > c D^{s_{0}}$ for some absolute constant $c>0$.
This is a contradiction if $s_{0}$ is sufficiently large.
\end{proof}
Using \eqref{eq:sep-inside} and \eqref{eq:separated-smalltiles} for every $J\in\calJ$, we obtain
\[
\sup_{y\in \frac12 J} \abs{T_{\TT_{2}\setminus\TS}^{*}g_{2}(y)}
\leq
\sup_{y\in \frac12 J} \sum_{s=s(J)}^{s(J)+s_{0}} \sum_{\Tp\in\TP : s(\Tp)=s} \abs{T_{\Tp}^{*}g_{2}(y)}
\lesssim
(s_{0}+1) \inf_{J} M g_{2}.
\]
Using \eqref{eq:tree*-loc-max} together with this fact, we obtain
\begin{equation}
\label{eq:tree2*-loc-max}
\begin{split}
\sup_{y\in (1+1/D) J} \abs{T_{\TT_{2}\cap\TS}^{*} g_{2}(y)}
&\leq
\inf_{y\in \frac12 J} \abs{T_{\TT_{2}\cap\TS}^{*}g_{2}(y)}
+
C \inf_{y\in J} Mg_{2}(y)\\
&\leq
\inf_{y\in \frac12 J} \abs{T_{\TT_{2}}^{*}g_{2}(y)}
+
\sup_{y\in \frac12 J} \abs{T_{\TT_{2}\setminus\TS}^{*}g_{2}(y)}
+
C \inf_{y\in J} Mg_{2}(y)\\
&\leq
\inf_{y\in \frac12 J} \abs{T_{\TT_{2}}^{*}g_{2}(y)}
+
C \inf_{y\in J} Mg_{2}(y)
\end{split}
\end{equation}
for $J\in\calJ$.
Using \eqref{eq:tree*-Hoelder} with $\TT=\TT_{1}$ and $\TT=\TT_{2}\cap\TS$, \eqref{eq:tree*-loc-max} with $\TT=\TT_{1}$, and \eqref{eq:tree2*-loc-max}, we obtain the estimate
\[
\abs{h_{J}(y)-h_{J}(y')}
\lesssim
\bigl(\frac{\abs{y-y'}}{\ell(J)}\bigr)^{\tau}
\prod_{j=1,2} \Bigl( \inf_{\frac12 J} \abs{T_{\TT_{j}}^{*} g_{j}} + \inf_{J} Mg_{j} \Bigr)
\]
for $y,y'\in I_{0} \cap (1+\frac1D) J$ and the functions
\[
h_{J}(y) :=
\chi_{J}(y)
\bigl( e(Q_{\TT_{1}}(0)-Q_{\TT_{1}}(y)) T_{\TT_{1}}^{*} g_{1}(y) \bigr)\\
\cdot
\overline{ \bigl( e(Q_{\TT_{2}}(0)-Q_{\TT_{2}}(y)) T_{\TT_{2}\cap\TS}^{*} g_{2}(y) \bigr) }.
\]
Moreover, since the function $T_{\TT_{1}}^{*} g_{1}$ is continuous and vanishes outside $I_{0}$, while $h_{J}$ vanishes outside $(1+\frac1D)J$, the same H\"older type estimate continues to hold for all $y,y'\in\R^{\ds}$.
Using \eqref{eq:leaf-separation} and Lemma~\ref{lem:osc-int}, this allows us to estimate
\begin{align*}
\abs[\Big]{\int_{\R^{\ds}} T_{\TT_{1}}^* g_{1} \overline{ T_{\TT_{2}\cap\TS}^* g_{2}} }
&\leq
\sum_{J} \abs[\Big]{\int e(Q(y)-Q(0)) h_{J}(y) \dif y}\\
&\lesssim
\sum_{J} \Delta_{J}^{-\tau/d} \meas{J} \prod_{j=1,2} \inf_{\frac12 J} \Bigl( \abs{T_{\TT_{j}}^{*} g_{j}} + Mg_{j} \Bigr)\\
&\lesssim
\Delta^{-(1-\eta)\tau/d}\int_{I_{0}} \prod_{j=1,2} \Bigl( \abs{T_{\TT_{j}}^{*} g_{j}} + Mg_{j} \Bigr)\\
&\leq
\Delta^{-(1-\eta)\tau/d}\prod_{j=1,2} \norm{ \abs{T_{\TT_{j}}^{*} g_{j}} + Mg_{j} }_{L^{2}(I_{0})}.
\end{align*}
It remains to consider the contribution of $\TT_{2}\setminus\TS$.
Let $\calJ' := \Set{J\in\calJ(\TT_{1}) \given J\subset I_{0}}$.
We claim that, for some $s_{\Delta}$ with $D^{s_{\Delta}} \sim \Delta^{\eta/d}$, we have
\begin{equation}
\label{eq:sep-smallT2-T1}
\Tp\in\TT_{2}\setminus\TS, J\in\calJ', I_{\Tp}^{*} \cap J \neq \emptyset
\implies
s(\Tp) \leq s(J) - s_{\Delta}.
\end{equation}
Indeed, if $s(\Tp) > s(J) - s_{\Delta}$, then $C D^{s_{\Delta}} I_{\Tp} \supset 100D\hat{J} \supset I_{\Tp'}$ for some $\Tp'\in\TT_{1}$, and, by Lemma~\ref{lem:normQ}, we obtain
\begin{multline*}
\Delta^{1-\eta}
>
\norm{Q}_{I_{\Tp}}
\gtrsim
D^{-d s_{\Delta}} \norm{Q}_{C D^{s_{\Delta}} I_{\Tp}}
\gtrsim
D^{-d s_{\Delta}} \norm{Q}_{100 D \hat{J}}\\
\geq
D^{-d s_{\Delta}} \norm{Q}_{I_{\Tp'}}
\geq
D^{-d s_{\Delta}} (\Delta - 5),
\end{multline*}
so that $1 > c D^{-d s_{\Delta}} \Delta^{\eta}$.
This is a contradiction if the proportionality constants in $D^{s_{\Delta}} \sim \Delta^{\eta/d}$ are chosen appropriately.
Using Lemma~\ref{lem:tree} and \eqref{eq:sep-smallT2-T1}, we obtain
\begin{align*}
\abs[\Big]{\int_{\R^{\ds}} T_{\TT_{1}}^* g_{1} \overline{ T_{\TT_{2}'}^* g_{2}} }
&\lesssim
\norm{g_{1}}_{2} \norm{P_{\calJ'} \abs{T_{\TT_{2}'}^* g_{2}} }_{2}\\
&\leq
\norm{g_{1}}_{2} \sum_{s\geq s_{\Delta}} \Bigl( \sum_{J\in\calJ'} \meas{J}^{-1} \abs[\Big]{\int_{J} \sum_{\Tp\in\TT_{2}' : s(\Tp) = s(J)-s, I_{\Tp}^{*} \cap J \neq \emptyset} T_{\Tp}^{*}g_{2} }^{2} \Bigr)^{1/2}\\
&\lesssim
\norm{g_{1}}_{2} \sum_{s\geq s_{\Delta}} \Bigl( \sum_{J\in\calJ'} \meas{J}^{-1} \abs[\Big]{\int_{J} M g_{2} \sum_{I\in\calD_{s(J)-s} : I\cap I_{0} = \emptyset, 2I \cap J \neq \emptyset} \one_{2I} }^{2} \Bigr)^{1/2}\\
&\lesssim
\norm{g_{1}}_{2} \sum_{s\geq s_{\Delta}} \Bigl( \sum_{J\in\calJ'} \Bigl(\int_{J} (M g_{2})^{2} \Bigr) \frac{\int_{J} \bigl( \sum_{I\in\calD_{s(J)-s} : I \cap J = \emptyset, 2I \cap J \neq \emptyset} \one_{2I}\bigr)^{2}}{\meas{J}} \Bigr)^{1/2}\\
&\lesssim
\norm{g_{1}}_{2} \sum_{s\geq s_{\Delta}} \Bigl( \sum_{J\in\calJ'} \Bigl(\int_{J} (M g_{2})^{2} \Bigr) \frac{D^{s(J)-s+s(J)(\ds-1)}}{D^{s(J) \ds}} \Bigr)^{1/2}\\
&\leq
\norm{g_{1}}_{2} \sum_{s\geq s_{\Delta}} D^{-s/2} \norm{M g_{2}}_{L^{2}(I_{0})}\\
&\lesssim
\Delta^{-\eta/(2d)} \norm{g_{1}}_{2} \norm{M g_{2}}_{L^{2}(I_{0})}.
\end{align*}
Choosing $\eta = 2\tau/(2\tau+1)$ and observing that $\norm{g_{1}}_{2} \leq \norm{Mg_{1}}_{L^{2}(I_{0})}$, we obtain the claim \eqref{eq:sep-trees} with $\epsilon = \tau/(d(2\tau+1))$.
\end{proof}

\subsection{Rows}
\begin{definition}
\label{def:row}
A \emph{row} is a union of normal trees with tops that have pairwise disjoint spatial cubes.
\end{definition}

\begin{lemma}[Row estimate]\label{lem:row}
Let $\TR_{1}$, $\TR_{2}$ be rows such that the trees in $\TR_{1}$ are $\Delta$-separated from the trees in $\TR_{2}$.
Then, for any $g_{1},g_{2}\in L^{2}(\R^{\ds})$, we have
\[
\abs[\Big]{ \int T_{\TR_{1}}^* g_{1} \overline{T_{\TR_{2}}^* g_{2}} }
\lesssim
\Delta^{-\epsilon} \norm{g_{1}}_{2} \norm{g_{2}}_{2}.
\]
\end{lemma}

\begin{proof}
The operators $S_{\TT}g := \abs{T_{\TT}^{*}g}+Mg$ are bounded on $L^{2}(\R^{\ds})$ uniformly in $\TT$ by Lemma~\ref{lem:tree} and the Hardy--Littlewood maximal inequality.
Using Lemma~\ref{lem:sep-trees}, we estimate
\begin{align*}
\abs[\Big]{ \int T_{\TR_{1}}^* g_{1} \overline{T_{\TR_{2}}^* g_{2}} }
&\leq
\sum_{\TT_{1}\in\TR_{1}, \TT_{2}\in\TR_{2}} \abs[\Big]{ \int T_{\TT_{1}}^* g_{1} \overline{T_{\TT_{2}}^* g_{2}} }\\
&=
\sum_{\TT_{1}\in\TR_{1}, \TT_{2}\in\TR_{2}} \abs[\Big]{ \int T_{\TT_{1}}^* (\one_{I_{\TT_{1}}} g_{1}) \overline{T_{\TT_{2}}^* (\one_{I_{\TT_{2}}} g_{2})} }\\
&\lesssim
\Delta^{-\epsilon} \sum_{\TT_{1}\in\TR_{1}, \TT_{2}\in\TR_{2}} \prod_{j=1,2} \norm{S_{\TT_{j}} \one_{I_{\TT_{j}}} g_{j}}_{L^{2}(I_{\TT_{1}} \cap I_{\TT_{2}})}\\
&\leq
\Delta^{-\epsilon}
\prod_{j=1,2} \Big(\sum_{\TT_{1}\in\TR_{1}, \TT_{2}\in\TR_{2}} \norm{S_{\TT_{j}} \one_{I_{\TT_{j}}} g_{j}}_{L^{2}(I_{\TT_{1}} \cap I_{\TT_{2}})}^{2} \Big)^{1/2}\\
&\leq
\Delta^{-\epsilon}
\prod_{j=1,2} \Big(\sum_{\TT_{j}\in\TR_{j}} \norm{S_{\TT_{j}} \one_{I_{\TT_{j}}} g_{j}}_{L^{2}(I_{\TT_{j}})}^{2} \Big)^{1/2}\\
&\lesssim
\Delta^{-\epsilon}
\prod_{j=1,2} \Big(\sum_{\TT_{j}\in\TR_{j}} \norm{\one_{I_{\TT_{j}}} g_{j}}_{L^{2}(\R^{\ds})}^{2} \Big)^{1/2}\\
&\leq
\Delta^{-\epsilon}
\norm{g_{1}}_{2}
\norm{g_{2}}_{2}.
\qedhere
\end{align*}
\end{proof}

\subsection{Forest estimate}
Recall our decomposition \eqref{eq:tree-dec} of the set of all tiles.
In view of Proposition~\ref{prop:sf}, it remains to estimate the contribution of the normal trees
\[
\TN_{n,k,j,l} := \TT_{n,k,j,l} \setminus \bd(\TT_{n,k,j,l})
\]
These sets are indeed (convex) trees, since $\bd(\TT)$ are up-sets (recall the definition \eqref{eq:T-boundary}).
\begin{proposition}
\label{prop:Fef-forest}
Let $\TF_{n,k,j}' := \cup_{l} \TN_{n,k,j,l}$.
Then
\begin{equation}
\label{eq:Fef-forest-global}
\norm{T_{\TF_{n,k,j}'}}_{2\to 2} \lesssim 2^{-n/2}.
\end{equation}
Assuming in addition \eqref{eq:tree-loc} for all $\Tp\in\TF_{n,k,j}'$, we obtain
\begin{equation}
\label{eq:Fef-forest-local}
\norm{T_{\TF_{n,k,j}'} \one_{F}}_{2\to 2} \lesssim \kappa^{\alpha} 2^{-n \epsilon}
\end{equation}
for any $0\leq \alpha < 1/2$.
\end{proposition}
\begin{proof}
We subdivide $\TF_{n,k,j}'$ into rows by the following procedure: for each $m\geq 0$, let inductively $\TR_{n,k,m} = \cup_{l\in L(k,m)}\TN_{n,k,j,l}$ be the union of a maximal set of trees whose spatial cubes are disjoint and maximal among those that have not been selected yet.
This procedure terminates after $O(2^{n} \log(n+1))$ steps, because the tree top cubes have overlap bounded by $O(2^{n} \log(n+1))$.
Applying Corollary~\ref{cor:tree} with the set $F$ and with the set $F$ replaced by $\R^{\ds}$ to each tree, we obtain
\[
\norm{T_{\TN_{n,k,j,l}} \one_{F}}_{2\to 2} \lesssim \kappa^{1/2} 2^{-n/2},
\quad
\norm{T_{\TN_{n,k,j,l}}}_{2\to 2} \lesssim 2^{-n/2}.
\]
Using normality of the trees and disjointness of their top cubes, we obtain
\begin{equation}
\label{eq:row-est}
\norm{T_{\TR_{n,k,m}} \one_{F}}_{2\to 2} \lesssim \kappa^{1/2} 2^{-n/2},
\quad
\norm{T_{\TR_{n,k,m}}}_{2\to 2} \lesssim 2^{-n/2}.
\end{equation}
Using the fact that
\begin{equation}
\label{eq:rows-disjoint}
T_{\TR_{n,k,m}}^{*} T_{\TR_{n,k,m'}} = 0
\text{ for } m \neq m'
\end{equation}
due to disjointness of $E(\Tp)$ for tiles that belong to separated trees, as well as Lemma~\ref{lem:row} and an orthogonality argument, we obtain \eqref{eq:Fef-forest-global}.

Using \eqref{eq:rows-disjoint} and \eqref{eq:row-est} gives
\begin{align*}
\norm{T_{\TF'_{n,k,j}} \one_{F} f}_{2}
&=
\big( \sum_{m \lesssim 2^{n} \log(n+1)} \norm{T_{\TR_{n,k,m}} \one_{F} f}_{2}^{2} \big)^{1/2}\\
&\lesssim
\big( \sum_{m \lesssim 2^{n} \log(n+1)} (\kappa^{1/2} 2^{-n/2} \norm{f}_{2})^{2} \big)^{1/2}\\
&\lesssim
\kappa^{1/2} 2^{-n/2} \norm{f}_{2} (2^{n} \log(n+1))^{1/2}\\
&\lesssim
\kappa^{1/2} (\log(n+1))^{1/2} \norm{f}_{2}.
\end{align*}
Taking a geometric average with \eqref{eq:Fef-forest-global}, we obtain \eqref{eq:Fef-forest-local}.
\end{proof}

\subsection{Orthogonality between stopping generations}
\label{sec:orth-stop-gen}
\begin{lemma}\label{lem:T*-stopping-cut}
Let $\TT \subset \TP_{k}$ be a tree and $k'>k$.
Then
\[
\norm{T_{\TT} \one_{F_{k'}}}_{2\to 2}
\lesssim
e^{-(k'-k)},
\]
where $F_{k'} = \cup_{F\in\calF_{k'}} F$.
\end{lemma}

\begin{proof}
Let $\calJ := \calJ(\TT)$ and $J\in\calJ$, so that $100D \hat{J} \supseteq I_{\Tp}$ for some $\Tp\in\TT$.

Let $F'\in \calF_{k+1}$ be such that $J\cap F' \neq\emptyset$.
Suppose that $\scale(F') \geq \scale(J) + 4$.
Then $(1+\frac1D) F' \supset 100D \hat{J} \supseteq I_{\Tp}$ and $\scale(F') > \scale(\Tp)$.
By part \ref{it:stopping-neighbors} of Lemma~\ref{lem:spatial-decomposition}, this implies $I\in\calF_{k+1}$ for some $I\supseteq I_{\Tp}$, contradicting $I_{\Tp}\in \calC_{k}$.

Therefore, we must have $\scale(F') \leq \scale(J) + 3$, and it follows that
\[
\sum_{F'\in\calF_{k+1} : J\cap F' \neq\emptyset} \abs{F'} \lesssim \abs{J}.
\]
Hence,
\begin{multline*}
\abs{J \cap F_{k'}}
\leq
\sum_{F'\in\calF_{k+1} : J\cap F' \neq\emptyset} \abs{F' \cap F_{k'}}\\
\lesssim
\sum_{F'\in\calF_{k+1} : J\cap F' \neq\emptyset} e^{-2(k'-k-1)} \abs{F'}
\lesssim
e^{-2(k'-k)} \abs{J}.
\end{multline*}
This implies $\norm{P_{\calJ} \one_{F_{k'}}}_{2\to 2} \lesssim e^{-(k'-k)}$, and the claim follows from Lemma~\ref{lem:tree}.
\end{proof}

\begin{proposition}
\label{prop:Fef-forest:small-support}
For any measurable subset $F'\subset \R^{\ds}$, we have
\begin{align}
\norm{T_{\TF_{n,k,j}'}^{*} T_{\TF_{n,k',j'}'}}_{2 \to 2} &\lesssim 10^{n} e^{-\abs{k-k'}}, \label{eq:gen-T*T}\\
\norm{T_{\TF_{n,k,j}'} \one_{F'} T_{\TF_{n,k',j'}'}^{*}}_{2 \to 2} &\lesssim 10^{n} e^{-\abs{k-k'}}. \label{eq:gen-TT*}
\end{align}
\end{proposition}

\begin{proof}
Let $\TR_{n,k,m}$ be the rows defined in the proof of Proposition~\ref{prop:Fef-forest}.
It suffices to show
\begin{align}
\norm{T_{\TR_{n,k,m}}^{*} T_{\TR_{n,k',m'}}}_{2 \to 2} &\lesssim e^{-\abs{k-k'}}, \label{eq:gen-T*T:row}\\
\norm{T_{\TR_{n,k,m}} \one_{F'} T_{\TR_{n,k',m'}}^{*}}_{2 \to 2} &\lesssim e^{-\abs{k-k'}}. \label{eq:gen-TT*:row}
\end{align}

Without loss of generality we may assume $k'\geq k$.
We will use the fact that
\[
T_{\TR_{n,k',m'}}
=
\one_{F_{k'}} T_{\TR_{n,k',m'}}
=
T_{\TR_{n,k',m'}} \one_{F_{k'}}
\]
with $F_{k'} = \cup_{F\in\calF_{k'}}F$ (the last equality uses normality of the trees).

Using \eqref{eq:row-est}, we estimate
\begin{align*}
LHS\eqref{eq:gen-T*T:row}
&=
\norm{ T_{\TR_{n,k,m}}^{*} \one_{F_{k'}} T_{\TR_{n,k',m'}} }_{2\to 2}\\
&\leq
\norm{ T_{\TR_{n,k,m}}^{*} \one_{F_{k'}} }_{2\to 2} \norm{ T_{\TR_{n,k',m'}} }_{2\to 2}\\
&\lesssim
\norm{ \one_{F_{k'}} T_{\TR_{n,k,m}} }_{2\to 2}.
\end{align*}
As a consequence of \eqref{eq:Lie-support-decay}, we have
\[
\norm{P_{\calL(\TN_{n,k,j,l})} \one_{F_{k'}}}_{2\to 2} \lesssim e^{-\abs{k-k'}},
\]
and \eqref{eq:gen-T*T:row} follows from Lemma~\ref{lem:tree}.
Similarly,
\begin{align*}
LHS\eqref{eq:gen-TT*:row}
&=
\norm{ T_{\TR_{n,k,m}} \one_{F'} \one_{F_{k'}} T_{\TR_{n,k',m'}}^{*} }_{2\to 2}\\
&\leq
\norm{ T_{\TR_{n,k,m}} \one_{F'\cap F_{k'}} }_{2\to 2} \norm{ T_{\TR_{n,k',m'}}^{*} }_{2\to 2}\\
&\lesssim
\norm{ T_{\TR_{n,k,m}} \one_{F_{k'}} }_{2\to 2}\\
&\lesssim
e^{-\abs{k-k'}}
\end{align*}
by Lemma~\ref{lem:T*-stopping-cut}.
\end{proof}

\section{Proof of Theorem~\ref{thm:loc}}
\label{sec:proof-loc}
As previously mentioned in Section~\ref{sec:discretization}, in view of Lemma~\ref{lem:M-loc}, we may replace the operator \eqref{eq:T} by \eqref{eq:T-smooth-trunc}, which in turn can be replaced by $T_{\TP}$.
\begin{proof}[Proof of \eqref{eq:loc:full}]
Using the decomposition \eqref{eq:tree-dec}, we split
\begin{multline*}
\norm{T_{\TP}}_{2\to 2}
\leq
\sum_{n=1}^{\infty} \sum_{j=1}^{C n^{2}} \Big( \norm[\big]{ \sum_{k\in\N} T_{\TF_{n,k,j}'}}_{2\to 2}
+ \norm[\big]{ \sum_{k\in\N} T_{\TA_{n,k,j}}}_{2\to 2}\\
+ \norm[\big]{ \sum_{k\in\N} \sum_{l} T_{\bd(\TT_{n,k,j,l})}}_{2\to 2} \Big).
\end{multline*}
The contribution of the last two summands is estimated by Proposition~\ref{prop:sf}.
In the first summand, we split the summation over $k$ in congruence classes modulo $C n$ and use Propositions \ref{prop:Fef-forest}, \ref{prop:Fef-forest:small-support}, and the Cotlar--Stein Lemma (see e.g.~\cite[Section VII.2]{MR1232192}).
\end{proof}

In the remaining part of the proof, we may assume $0 < \nu,\kappa < 1$.
Indeed, the cases $\nu=0$ and $\kappa=0$ are trivial, and in the cases $\nu \geq 1$ or $\kappa \geq 1$ the respective estimates \eqref{eq:loc:G} and \eqref{eq:loc:F} follow from \eqref{eq:loc:full}.

\begin{proof}[Proof of \eqref{eq:loc:G}]
Let $\TP_{\tilde G} := \Set{ \Tp\in\TP \given I_{\Tp}^{*} \subseteq \tilde{G}}$, then $T_{\Tp}\one_{\R^{\ds}\setminus \tilde{G}}=0$ if $\Tp\in\TP_{\tilde G}$.
Hence
\[
\norm{ \one_{G} T_{\TP} \one_{\R^{\ds} \setminus \tilde{G}}}_{2 \to 2}
=
\norm{ \one_{G} T_{\TP \setminus \TP_{\tilde G}} \one_{\R^{\ds} \setminus \tilde{G}}}_{2 \to 2}
\leq
\norm{ \one_{G} T_{\TP \setminus \TP_{\tilde G}}}_{2 \to 2}.
\]
In order to estimate the latter quantity, we run the proof of \eqref{eq:loc:full} with $\TP$ replaced by $\TP\setminus\TP_{\tilde{G}}$ and (formally) $\smax(x) = -\infty$ for $x\in\R^{\ds}\setminus G$.

The main change is that all tiles now have density $2^{n} \lesssim \nu$.
This yields the required improvement in the estimate for the main term.
In the error terms, we use Proposition~\ref{prop:sf-loc} with $F=\R^{\ds}$.
The hypothesis \eqref{eq:sf-loc:assume} is satisfied, because we have removed all tiles whose spatial cubes are contained in $\tilde{G}$.
\end{proof}

\begin{proof}[Proof of \eqref{eq:loc:F}]
Let $\TP_{\tilde F} := \Set{ \Tp\in\TP \given I_{\Tp} \subseteq \tilde{F}}$, then $\one_{\R^{\ds}\setminus \tilde{F}} T_{\Tp}=0$ if $\Tp\in\TP_{\tilde F}$.
Hence
\[
\norm{ \one_{\R^{\ds} \setminus \tilde{F}} T_{\TP} \one_{F}}_{2 \to 2}
=
\norm{ \one_{\R^{\ds} \setminus \tilde{F}} T_{\TP \setminus \TP_{\tilde F}} \one_{F}}_{2 \to 2}
\leq
\norm{ T_{\TP \setminus \TP_{\tilde F}} \one_{F}}_{2 \to 2}.
\]
In order to estimate the latter term, we again run the proof of \eqref{eq:loc:full} with $\TP$ replaced by $\TP \setminus \TP_{\tilde F}$.
In particular, we split
\begin{multline*}
\norm{T_{\TP} \one_{F}}_{2\to 2}
\leq
\sum_{n=1}^{\infty} \sum_{j=1}^{C n^{2}} \Big( \norm[\big]{ \sum_{k\in\N} T_{\TF_{n,k,j}'} \one_{F}}_{2\to 2}
+ \norm[\big]{ \sum_{k\in\N} T_{\TA_{n,k,j}} \one_{F}}_{2\to 2}\\
+ \norm[\big]{ \sum_{k\in\N} \sum_{l} T_{\bd(\TT_{n,k,j,l})} \one_{F}}_{2\to 2} \Big).
\end{multline*}
The contribution of the last two terms is taken care of by Proposition~\ref{prop:sf-loc} with $G=\R^{\ds}$.
In the estimate for the main term, we use \eqref{eq:Fef-forest-local} in place of \eqref{eq:Fef-forest-global} and split the summation over $k$ in congruence classes modulo $\lceil C n (\abs{\log \kappa} + 1) \rceil$.
\end{proof}

\appendix
\section{A van der Corput type oscillatory integral estimate}
\label{sec:vdC}
We will use the following van der Corput type estimate for oscillatory integrals in $\R^{\ds}$ that refines \cite[Proposition 2.1]{MR1879821}.

\begin{lemma}
\label{lem:osc-int}
Let $\psi : \R^{\ds}\to\C$ be a measurable function with $\supp\psi\subset J$ for a cube $J$.
Then, for every $Q\in\calQ_{d}$, we have
\[
\abs[\big]{\int_{\R^{\ds}} e(Q(x)) \psi(x) \dif x}
\lesssim
\sup_{\abs{y}<\Delta^{-1/d} \ell(J)} \int_{\R^{\ds}} \abs{\psi(x)-\psi(x-y)} \dif x,
\quad
\Delta = \norm{Q + \R}_{J}+1.
\]
\end{lemma}
\begin{remark}
The supremum over $\abs{y}<\Delta^{-1/d} \ell(J)$ above can be replaced by an average.
\end{remark}
\begin{proof}
By scaling and translation, we may assume $\ell(J)\sim 1$ and $J\subset B(0,1/2)$.
Let $\beta$ denote the right-hand side of the conclusion.
If $\Delta \lesssim 1$, then $\norm{\psi}_{1} \lesssim \beta$, so the result is only non-trivial if $\Delta \gg 1$.
In this case, we replace $\psi$ on the left-hand side by $\tilde\psi:=\phi*\psi$, where $\phi = \Delta^{\ds/d} \phi_{0}(\Delta^{1/d} \cdot)$ and $\phi_{0}$ is a smooth positive bump function with integral $1$ supported on the unit ball.
The error term is controlled by
\begin{align*}
\int \abs{\psi-\tilde\psi}(x) \dif x
&=
\int \abs[\big]{\int (\psi(x)-\psi(x-y))\phi(y) \dif y} \dif x\\
&\leq
\int \phi(y) \int \abs{\psi(x)-\psi(x-y)} \dif x \dif y\\
&\lesssim
\beta.
\end{align*}
Moreover, $\supp \tilde\psi \subseteq B(0,1)$ and
\begin{align*}
\int \abs{\partial_{i} \tilde\psi(x)} \dif x
&=
\int \abs{\int \psi(x-y) \partial_{i} \phi(y) \dif y} \dif x\\
&=
\int \abs{\int (\psi(x)-\psi(x-y)) \partial_{i} \phi(y) \dif y} \dif x\\
&\leq
\int \int \abs{\psi(x)-\psi(x-y)} \abs{\partial_{i} \phi(y)} \dif y \dif x\\
&\lesssim
\Delta^{\ds/d+1/d}\int \int_{B(0,\Delta^{-1/d})} \abs{\psi(x)-\psi(x-y)} \dif y \dif x\\
&\lesssim
\Delta^{1/d} \beta
\end{align*}
for every $i=1,\dotsc,\ds$.
The result now follows from the proof of \cite[Proposition 2.1]{MR1879821} applied to $\tilde\psi$.
Notice that the one-dimensional van der Corput estimate (Corollary on p.~334 of \cite{MR1232192}) used in that proof only requires an estimate on the integral of $\nabla \tilde\psi$.
\end{proof}

\section{The extrapolation argument}
\label{sec:extrapolation}
Theorem~\ref{thm:main} is deduced from Theorem~\ref{thm:loc} using Bateman's extrapolation argument that first appeared in \cite{MR3148061} (see also \cite[Theorem 1.1]{MR3352435} and \cite[Theorem 2.27]{MR3841536} for an abstract formulation of this argument).
In order to keep our exposition self-contained, we present this argument in the case needed here.
\begin{lemma}
\label{lem:extrapolation}
Let $1<p,q<\infty$, $(X,\mu)$ be a $\sigma$-finite measure space, and $g : X \to \C$ be a measurable function.
Suppose that there exists $A < \infty$ such that, for every measurable subset $G \subset X$ with $0 < \mu(G) < \infty$, there exists a measurable subset $\tilde{G} \subset X$ with $\mu(\tilde{G}) \leq \mu(G)/2$ such that $\norm{g \one_{G \setminus \tilde{G}}}_{L^{q,\infty}(X)} \leq A \mu(G)^{1/q-1/p}$.
Then $\norm{g}_{L^{p,\infty}(X)} \lesssim A$.
\end{lemma}
\begin{proof}
Let $G_{0} \subset X$ with $\mu(G_{0}) < \infty$ be given.
For $n\in \Set{1,2,\dotsc}$, define inductively $G_{n+1} := \widetilde{G_{n}}$, so that $\mu(G_{n}) \leq 2^{-n} \mu(G_{0})$.
Then
\begin{align*}
\int_{G_{0}} \abs{g} \dif\mu
&=
\sum_{n=0}^{\infty} \int_{X} \abs{g \one_{G_{n}\setminus G_{n+1}}} \dif\mu\\
&\leq
\sum_{n=0}^{\infty} \norm{g \one_{G_{n}\setminus G_{n+1}}}_{L^{q,\infty}} \norm{\one_{G_{n}}}_{L^{q',1}}\\
&\lesssim
A \sum_{n=0}^{\infty} \mu(G_{n})^{1/q-1/p} \mu(G_{n})^{1/q'}\\
&\leq
A \sum_{n=0}^{\infty} (2^{-n} \mu(G_{0}))^{1/p'}\\
&\lesssim
A \mu(G_{0})^{1/p'}.
\end{align*}
By duality between $L^{p,\infty}$ and $L^{p',1}$, this implies the claim.
\end{proof}

\begin{proof}[Proof of Theorem~\ref{thm:main} assuming Theorem~\ref{thm:loc}]
By standard real interpolation theory \cite{MR0482275}, it suffices to show that $T$ is a bounded operator from $L^{p,1}(\R^{\ds})$ to $L^{p,\infty}(\R^{\ds})$ for every $1<p<\infty$.
To see this, let $F \subset \R^{\ds}$ be a measurable subset with $0 < \abs{F} < \infty$, $f : \R^{\ds} \to \C$ a measurable function with $\abs{f} \leq \one_{F}$, and $g := Tf$.
If $G \subset \R^{\ds}$ is a measurable subset with $0 < \abs{G} < \infty$, then, for a sufficiently large absolute constant $C$, the set $\tilde{G} := \Set{ M\one_{F} > C \abs{F} \abs{G}^{-1}}$ satisfies $\abs{\tilde{G}} \leq \abs{G}/2$.
On the other hand, by \eqref{eq:loc:F}, for every $0\leq \alpha < 1/2$, we have
\[
\norm{g \one_{G \setminus \tilde{G}}}_{L^{2,\infty}}
\leq
\norm{g \one_{\R^{\ds} \setminus \tilde{G}}}_{2}
\lesssim_{\alpha}
(\abs{F} \abs{G}^{-1})^{\alpha} \norm{f}_{2}
\leq
\abs{F}^{\alpha+1/2} \abs{G}^{-\alpha}.
\]
Using this with $\alpha = 1/p-1/q$ for $1 < p \leq 2 = q$, we see that the hypothesis of Lemma~\ref{lem:extrapolation} holds for the function $g$ with $A \lesssim_{p} \abs{F}^{1/p}$.
Hence, by Lemma~\ref{lem:extrapolation}, we obtain $\norm{g}_{L^{p,\infty}} \lesssim \abs{F}^{1/p}$.
This shows that $T$ is a bounded operator from $L^{p,1}$ to $L^{p,\infty}$.

In the case $2<p<\infty$, we can run the above argument for the adjoint operator $T^{*}$ in place of $T$, using the estimate \eqref{eq:loc:G} in place of \eqref{eq:loc:F}.
\end{proof}

Finally, in Section~\ref{sec:proof-loc}, we have used a localized estimate for the Hardy--Littlewood maximal operator.
We include the short proof.
\begin{lemma}
\label{lem:M-loc}
Let $0\leq \alpha < 1/2$ and $0 < \nu \leq 1$.
Let $G\subset\R^{\ds}$ be a measurable subset and $\tilde G := \Set{ M \one_{G} > \nu }$.
Then
\[
\norm{\one_{G} M \one_{\R^{\ds} \setminus \tilde{G}}}_{2\to 2} \lesssim_{\alpha} \nu^{\alpha},
\quad
\norm{\one_{\R^{\ds} \setminus \tilde{G}} M \one_{G}}_{2\to 2} \lesssim_{\alpha} \nu^{\alpha}.
\]
\end{lemma}
\begin{proof}
By the Fefferman--Stein maximal inequality \cite{MR0284802}, we have
\[
\norm{\one_{G} M \one_{\R^{\ds} \setminus \tilde{G}} f}_{1,\infty}
=
\norm{M \one_{\R^{\ds} \setminus \tilde{G}} f}_{L^{1,\infty}(\one_{G})}
\lesssim
\norm{\one_{\R^{\ds} \setminus \tilde{G}} f}_{L^{1}(M \one_{G})}
\leq
\nu \norm{f}_{1}.
\]
Interpolating with the trivial $L^{\infty}$ estimate, we obtain the first claim.

Let now $q=1/\alpha$.
Then, by H\"older's inequality,
\[
M \one_{G} f
\leq
(M_{q} \one_{G}) (M_{q'} f)
=
(M \one_{G})^{\alpha} (M_{q'} f).
\]
Hence,
\[
\norm{\one_{\R^{\ds} \setminus \tilde{G}} M \one_{\tilde{G}} f}_{2}
\leq
\norm{\one_{\R^{\ds} \setminus \tilde{G}} (M\one_{G})^{\alpha} M_{q'} f}_{2}
\leq
\nu^{\alpha} \norm{M_{q'} f}_{2}
\lesssim
\nu^{\alpha} \norm{f}_{2},
\]
where we have used the fact that $M_{q'}$ is bounded on $L^{2}$ provided that $q'<2$.
\end{proof}

\printbibliography

\end{document}